\documentclass[11pt,reqno]{amsart}

\usepackage{amsmath}
\usepackage{amsthm,amsfonts,amsxtra,amssymb,amscd}
\usepackage{enumerate}
\usepackage{verbatim}

\theoremstyle{plain}
\newtheorem*{lemma*}{Lemma}
\newtheorem{lemma}[subsection]{Lemma}
\newtheorem*{theorem*}{Theorem}
\newtheorem{theorem}[subsection]{Theorem}
\newtheorem*{proposition*}{Proposition}
\newtheorem{proposition}[subsection]{Proposition}
\newtheorem*{corollary*}{Corollary}
\newtheorem{corollary}[subsection]{Corollary}
\theoremstyle{definition}
\newtheorem*{definition*}{Definition}
\newtheorem{definition}[subsection]{Definition}
\newtheorem*{example*}{Example}
\newtheorem{notation}[subsection]{Notation}

\theoremstyle{remark}
\newtheorem*{remark*}{Remark}
\newtheorem{remark}[subsection]{Remark}

\newcommand{\undera}{{\underline a}}
\newcommand{\aundera}{{a}}
\newcommand{\underb}{{\underline b}}
\newcommand{\underc}{{\underline c}}

\newcommand{\mQ}{\mathbb Q}

\newcommand{\mZ}{\mathbb Z}

\newcommand{\calA}{\mathcal A}
\newcommand{\calB}{\mathcal B} \newcommand{\calC}{\mathcal C}

 \newcommand{\calL}{\mathcal L} \newcommand{\calM}{\mathcal M}

\newcommand{\gob}{\mathfrak b}  
  \newcommand{\gog}{\mathfrak g}
\newcommand{\goh}{\mathfrak h}  
\newcommand{\gok}{\mathfrak k} \newcommand{\gol}{\mathfrak l} 
\newcommand{\gon}{\mathfrak n}  
  
\newcommand{\got}{\mathfrak t} \newcommand{\gou}{\mathfrak u}

\newcommand{\gra}{\alpha}  \newcommand{\grb}{\beta}       
\newcommand{\grd}{\delta}   
  \newcommand{\grl}{\lambda}     \newcommand{\grs}{\sigma}
\newcommand{\grf}{\varphi}       

\newcommand{\mk}  {\Bbbk}

\newcommand{\lra}      {\longrightarrow}

\renewcommand{\geq}    {\geqslant}
\renewcommand{\leq}    {\leqslant}
\newcommand{\senza}    {\smallsetminus}

\newcommand{\isocan}   {\simeq}
\DeclareMathOperator{\Hom}{Hom}

\DeclareMathOperator{\Spec}{Spec}

            \newcommand{\st}       {\, : \,}
       
         \newcommand{\mand}     {\text{ and }}

\title[Steinberg section]{A generalized Steinberg section and
  branching rules for quantum groups at roots of 1}

\author{C.~De Concini,\quad A.~Maffei}

\thanks{The  authors are partially supported by  MIUR}

\dedicatory{Dedicated to the memory of I.~M.~Gelfand}
\address{Dipartimento di Matematica ``G. Castelnuovo", Universit\`a ``La Sapienza", Roma}
\begin{document}

\maketitle

\begin{Small}
\begin{it}
\begin{flushright}
\parbox{4in}{
\begin{verse}
O cara piota mia che s\`\i\ t'insusi,\\
che come veggion le terreni menti,\\
non cap\`ere in triangol due ottusi,\\
\medskip
cos\`\i\ vedi le cose contingenti\\
anzi che sieno in s\'e, mirando il punto\\
a cui tutti li tempi son presenti;\\
\medskip
\emph{(La Divina Commedia, Paradiso, Canto XVII)}
\end{verse}}
\end{flushright}
\end{it}
\end{Small}

\begin{abstract} In this paper we construct a generalization of the classical Steinberg
section \cite{stein} for the quotient map of a semisimple group with
respect to the conjugation action. We then give various applications
of our result including the construction of a sort of Gelfand-Zeitlin basis for a generic irreducible representation of $U_q(GL(n))$
when $q$ is a primitive odd root of unity. \end{abstract}

\section{Introduction}
In his famous paper \cite{stein}, Steinberg introduces a remarkable
section of the quotient map of a simply connected semisimple group
modulo its conjugation action. In this paper we are going to extend
this construction as follows.  Let $G$ be a reductive group with
simply connected semisimple factor over an algebraically
closed  field $\mk$. Let $T$
be a maximal torus, $B\supset T$ a Borel subgroup and $B^-$ the
opposite Borel subgroup. We take a sequence $\mathcal
L=\{L_1\subset L_2\subset\cdots\subset L_h=G\}$ of standard Levi
subgroups such that for each $i=1,\ldots ,h-1$, each simple factor of
$L_i$ does not coincide with a simple factor of $L_{i+i}$.  A sequence
$\calM = \{M_1\subset \dots \subset M_h\}$ of connected subgroups of $G$
is  said to be admissible if, for all $i=1,\ldots ,h$, $M_i$ is a subgroup of
$L_i$ containing $L_i^{ss}$ and there exist subtori $S_i \subset T$
such that \begin{enumerate}[\indent S1:] \item $M_i = S_i \ltimes L_i^{ss}$;
  \item $S_{i+1} \cap M_i = \{1\}$. \end{enumerate} Notice that if $G$
is simple and $L$ is any standard Levi factor properly contained in
$G$, then $L\subset G$ is automatically admissible. Also the sequence
$\{GL(1)\subset GL(2)\cdots \subset GL(n)\}$ is admissible.

Taking the unipotent radical $U^-\subset B^-$, in Section 2 we define
a simultaneous quotient map
$$ r_\calM : B \times U^- \lra \prod_{i=1}^h M_i/\!/ Ad(M_i) $$ and we
show, Theorem \ref{teo:simultanea}, that $r_\calM$ has a section,
which if $\calM=\{G\}$, coincides with the Steinberg section once we identify $B \times U^- $ with the open Bruhat cell $BU^-\subset G$. 

In the case of the sequence $\{GL(1)\subset GL(2)\cdots \subset
GL(n)\}$ we give, in  Theorem \ref{teo:GL}, an alternative elementary proof of this statement
which uses a curious explicit parametrization of the Borel subgroup
$B\subset GL(n)$.

As a consequence of our result the coordinate ring $
\bigotimes_{i=1}^h \mk[M_i]^{Ad(M_i)}$ of $\prod_{i=1}^h M_i/\!/
Ad(M_i)$, which is a polynomial ring with some variables inverted,
embeds in the coordinate ring of $B\times U^-$.   Consider the  group 
$$ H=\{(x,y)\in B\times B^- \st \pi_T(x)=\pi^-_T(y)^{-1}\}$$ where
$\pi_T,\pi_T^-$ are the projections on $T$ of $B,B^-$ respectively, Poisson dual of $G$. Define $\rho:H\to G$ by
$\rho((x,y))=xy^{-1}$. The map $\rho$ is a $2^{rk G}$ to $1$ covering map of the open set $BU^-$. So, we
also get an inclusion of $ \bigotimes_{i=1}^h \mk[M_i]^{Ad(M_i)}$ into
$\mk[H]$ and we remark that the image is a Poisson commutative
subalgebra of $\mk[H]$.

This is particularly interesting in the case of the sequence
$\{GL(1)\subset GL(2)\cdots \subset GL(n)\}$. In this case this
algebra, which is a multiplicative analogue of the Gelfand-Zeitlin
subalgebra of the coordinate ring of the space of $n\times n$
matrices, has transcendence degree $n(n^2+n)/2$ which is maximal. So, it    gives a completely integrable Hamiltonian system. It is worth
pointing out that in the case of square matrices and also in the case
of the Poisson dual of the unitary groups these systems have been
known and studied for quite some time, see for example
\cite{AlMe,KW1,KW2} and reference therein.

Finally in the last two sections, we give an application of our result to
the study of the quantum enveloping algebra associated to $G$ when the
deformation parameter is a primitive root of one. We show, following
ideas in \cite{dprr}, how to decompose a ``generic" irreducible
representation when it is restricted to the quantum enveloping algebra
of an admissible subgroup $M$ satisfying some further assumptions.

In the case in which  $G=GL(n)$ and $M=GL(n-1)$, this gives a multiplicity
one decomposition which can be used to decomposes our module into a
direct sum of one dimensional subspaces in a way which is quite
analogue to that appearing in the classical work of Gelfand-Zeitlin
\cite{GZ}.

We wish to thank David Hernandez for various discussions about the
structure of representations of quantum affine algebras at roots of
one.  These discussions led us to consider the problem discussed in
Section 6 which has been  the starting point of this paper. 
We would like to thank also Ilaria Damiani for help regarding quantum groups and the Hausdorff Institute for Mathematics 
for its hospitality.

\section{A simultaneous version of Steinberg resolution}\label{sez:stein}
Let $G$ be a connected reductive group over an algebraically closed
field $\mk$ whose Lie algebra we denote by $\mathfrak g$. Assume   that the semisimple part of $G$ is simply connected. 
    Let $q:G\lra G/\!/ Ad(G)$ be
the quotient of $G$ under the adjoint action.

Choose  a maximal torus $T$ and a Borel subgroup $B\supset T$ in $G$.
Denote by $\Phi$ the corresponding set of roots, by $\Delta$ that of
simple roots, by $\Phi^+$ that of positive roots and by  $W=N_G(T)/T$  the
  Weyl group. The inclusion $T\subset G$ induces an
isomorphism $G/\!/Ad(G) \isocan T/W$.

If $\gra\in \Phi$,  we denote by $\check\gra(t)$ the corresponding
cocharacter. Given $\alpha$, we can associate to $\alpha$ an $SL(2)$
embedding $\gamma_{\gra}$ into $G$ and we take the unipotent one
parameter subgroups
$$X_\gra(a)=\gamma_{\gra}\Bigg (\left(\begin{array}{cc}1 & a \\0 & 1\end{array}\right)\Bigg),\ \
Y_\gra(a)=\gamma_{\gra}\Bigg(\left(\begin{array}{cc}1 & 0 \\a & 1\end{array}\right)\Bigg).$$
Finally set
$$s_\gra=X_{\gra}(1)Y_\gra(-1)X_\gra(1) =\gamma_{\gra}\Bigg
(\left(\begin{array}{cc}0 & 1 \\-1 & 0\end{array}\right)\Bigg).$$

\subsection{The Steinberg section in the semisimple case}\label{ssec:Steinbergss}
Let us recall the construction of the Steinberg section. 
Consider the vector space $\mk^{\Delta}$ of $\mk$-valued maps on the
set of simple roots $\Delta$. If we fix an order
$\gra_1,\ldots,\gra_n$ of the set of simple roots, we  identify
$\mk^{\Delta}$ with $\mk^n$ setting for each $\aundera\in
\mk^{\Delta}$, $a_i=\aundera(\alpha_i)$. In \cite{stein} , given such an ordering, 
Steinberg defined the map $St: \mk^n\to G$ by
  $$St((a_1,\dots,a_n))= X_{\gra_n}(a_n)s_{\gra_n}\cdots
X_{\gra_1}(a_1)s_{\gra_1}$$ and proved the following Theorem.

\begin{theorem}[\cite{stein}, Theorem 7.9] If $G$ is semisimple
and simply connected, the composition $q\circ St:\mk^n\to G/\!/ Ad(G)$
is an isomorphism.

Furthermore, the image of $St$ is contained in the regular locus and
intersects each regular conjugacy class exactly in one point.
\end{theorem}

Of course we can use the   identification of $\mk^{\Delta}$ with
$\mk^n$ associated to a chosen order, and  describe the
Steinberg section as a map $St:\mk^\Delta \lra G$ defined by
$St(\aundera)=X_{\gra_n}\big(a(\gra_n)\big)s_{\gra_n} \cdots
X_{\gra_1}\big(a(\gra_1)\big) s_{\gra_1}$.

We will need to compare two Steinberg sections constructed by
considering  different orders of the set of   simple roots. If $I\subset \Delta$,   we
identify $\mk^I$ with the subspace of $\mk^\Delta$ of functions
vanishing on $\Delta \senza I$. We set $\mk^{\gra}:=\mk^{\{\gra\}}$. 

 \begin{lemma}\label{lemma:riordinared} Let $St,St':\mk^\Delta \lra G$
be the Steinberg sections associate to two different orders of the
simple roots. Then there exists a morphism $g:\mk^\Delta \lra G$, an
element $w$ of the Weyl group  and an action of $T$ on
$\mk^\Delta$ such that for each $\gra \in \Delta$ the line
$\mk^{\gra}$ is stable by the action of $T$ and 
$$t\, St(\aundera)=g(t\cdot\aundera)w(t)St'(t\cdot\aundera)g(t\cdot\aundera)^{-1}$$
for all $t\in T$ and $\aundera \in \mk^\Delta$, where we denote by $w(t)$ the element $ntn^{-1}$, $n$ being a representative of $w$ in $N_G(T)$.
\end{lemma}

\begin{proof}
It is known (see for example \cite{HUM} pp.74,75) that two ordered
sequence of simple roots can be obtained from each other by
applying recursively the following two operations: commute the order
of two adjacent orthogonal simple roots $\gra_i, \gra_j$ and move the
last simple root, let's say $\grb$, to the first position. In the first case
$X_{\gra_i},s_{\gra_i}$ and $X_{\gra_j},s_{\gra_j}$ commute, so we can choose $g=1$, $w=1$ and the trivial action of $T$ on $\mk^\Delta$.
 In  the second case everything follows,   choosing  $g(\aundera)= X_{\grb}(\aundera(\grb))\,s_{\grb}$, $w=s_{\beta}$ and
$$(t\cdot a)(\gamma)=\begin{cases} a(\gamma) \text {\ \ \ \ \ if\ \ } \gamma \neq \beta\\ t^{\beta}a(\gamma)  \text {\ \ if\ \ } \gamma = \beta\end{cases}.$$
\end{proof}

Before we proceed, let us write the Steinberg section in a slightly
different way. Fix an order of the simple roots. Set $s_i:=s_{\gra_i}$
and define $\grb_i=s_n\cdots s_{i+1}(\gra_i)$ and $w=s_n\cdots
s_1$. The ordered set of roots $\beta_1,\ldots ,\beta_n$ will be
called the beta set of roots associated to the given order.

Notice for any $i,j$, $\beta_i+\beta_j$ is never a root. Indeed if
$i=j$ this is clear.  Otherwise assume that $i>j$.  Then
$\grb_i+\grb_j=s_n\cdots s_{i}(s_{i-1}\cdots
s_{j+1}(\gra_j)-\gra_i)$. But $s_{i-1}\cdots s_{j+1}(\gra_j)$ is a
positive root whose support does not contain $\alpha_i$, hence
$s_{i-1}\cdots s_{j+1}(\gra_j)-\gra_i$ and   also $\grb_i+\grb_j$ are not roots.

If $\aundera=(a_1,\ldots a_n) \in \mk^n$, we can then rewrite the
Steinberg section as
$$St(\aundera)=\left(\prod_{i=1}^nX_{\beta i}(a_i)\right)\,w =: X_G(\aundera)w.$$

\subsection{The Steinberg section in the reductive case}\label{ssec:Steinbergred}
We are now going to construct a section also in    the reductive case.  For
this we need a slight extension of Steinberg result which follows in
exactly the same way and whose proof we leave to the reader.

Let $G^{ss}$ be the commutator subgroup of $G$ (which by assumption is
simply connected) and let $T^{ss}=G^{ss}\cap T$, be a maximal
torus in $G^{ss}$.  The inclusion $T^{ss}\subset T$ induces a
surjection of character groups $\Lambda (T)\to \Lambda (T^{ss})$.
Splitting, we obtain a subtorus $S$ of $T$ such that $T = S \times T^{ss}$.  Then
$G=S \ltimes G^{ss}$. Furthermore denote the center of $G$ by $Z(G)$.


Let us recall that a basis for the ring of functions on $G^{ss}$
invariant under conjugation is given by the characters of the
irreducible $G^{ss}$-modules and a set of polynomial generators is given by
the characters of the fundamental representations.  Notice that, if we take
an irreducible module $V$ for $G^{ss}$, $G^{ss}\cap Z(G)$ acts on $V$ by
a character which can be extended (not uniquely) to a character of
$Z(G)$. Thus, in this way, $V$ becomes an irreducible $G$-module and its
character a central function on $G$. It follows that the restriction
map induces a surjective morphism $p: G\to G^{ss}/\!/Ad(G^{ss})$ which
restricted to $G^{ss}$ is the quotient map. Also, the projection $G\to
S$ commutes with the adjoint action and this gives an identification
of $G/\!/Ad(G)$ with $S \times G^{ss}/\!/Ad(G^{ss})$. Under this
identification, the quotient map $q:G = S \ltimes G^{ss} \to
G/\!/Ad(G)$ is given by $q((s,g))=(s,p( sg))$. We have the
following generalization of the Steinberg section to the reductive
case.

\begin{proposition}\label{prp:Steinbergred} Let $St : \mk^\Delta \lra G^{ss}$ be a
Steinberg section of the semisimple factor of $G$.  Consider the map
$St_S: S\times \mk^\Delta\to G$ given by
$$St_S(s,\aundera)=s\cdot St(\aundera).$$ Then the composition
$q\circ St_S$ is an isomorphism. Moreover, for each $t\in T$, the map
$\aundera \mapsto p(t\cdot St(\aundera))$ from $\mk^\Delta$ to
$G^{ss}/\!/Ad(G^{ss})$ is an isomorphism.
\end{proposition}

\begin{proof}
In order to show our claim, it clearly suffices to see that our
morphism is bijective. Hence,  it is enough to prove the second claim.
This follows immediately since, being $p$ invariant under the
adjoint action, this map is given by $p(X_{\gra_1}(a_1)s_1\cdots
X_{\gra_n}(a_n)s_n \cdot t)$ and $s_n\cdot t$ is just another
representative of the reflection $s_n$ in the Weyl group for which
Steinberg's proof can be repeated verbatim.
\end{proof}

\subsection{Simultaneous Steinberg section for the semisimple factors of a sequence of Levi subgroups}\label{ssec:ss}
At this point,  we have all the ingredients to construct a simultaneous Steinberg
section for a sequence of Levi subgroups. We need to introduce some
notations. If $L$ is a standard Levi subgroup of $G$ with respect to
the given choice of $T$ and $B$, we denote by $\Phi_L$ be its root system  
and by $\Delta_L=\Delta\cap \Phi_L$ its simple roots. Recall that in
Section \ref{ssec:Steinbergss} we have introduced the beta set
associated to an order of the simple roots. The following Lemma
allows us to choose an order of the simple roots such that its beta set has
some   particularly nice  properties with respect to the Levi factor $L$.

\begin{lemma}\label{lemma:compatibile}Let $L$ be a standard Levi subgroup of $G$ which
does not contain any simple factor of $G$. Then there exists an order
of the simple roots $\gra_1,\dots,\gra_n$ such that
\begin{enumerate}
\item $\Delta_L=\{\gra_1,\dots,\gra_m\}$.
\item For all $i=1,\dots,n$ the roots $\grb_i=s_n\cdots
  s_{i+1}(\gra_i) \not \in \Phi_L$.
\end{enumerate}
\end{lemma}

\begin{proof} We are going to construct on order $\gra_1,\dots,\gra_n$ on $\Delta$ with the property that $\gra_{m+1},\dots,\gra_n$ is an arbitrary order on   $\Delta\setminus\Delta_L$ and, for every $i\leq m$, 
  there is a $j>m$ such that $\langle s_m\cdots
s_{i+1}\alpha_i,\check \alpha_j\rangle\neq 0$.  
If this is the
case, this order obviously satisfies our first requirement. The second  is clearly satisfied  for $i>m$. For $i\leq m$,  taking the minimum   $j$ for which  $\langle s_m\cdots
s_{i+1}\alpha_i,\check \alpha_j\rangle\neq 0$, we get that the support of the root
$s_j\cdots s_{i+1}\alpha_i$ and hence of the root $\grb_i$ contains
$\alpha_j$.

At this point we can assume $m\geq 1$. By assumption, since $L$ does not contain any simple factor of $G$, there is a root $\alpha\in 
\Delta_L$ such that $\langle\alpha,\check \alpha_j\rangle\neq 0$ for some $j>m$. We set $\alpha_m=\alpha$ and consider the Levi subgroup 
$L'$ associated to $\Delta_L\setminus \{\alpha_m\}$. The assumptions of the Lemma are satisfied by $L'$, hence we may choose an order 
$\alpha_1,\ldots ,\alpha_{m-1}$ of $\Delta_L\setminus \{\alpha_m\}$ with the property that, for every $i\leq m-1$, there is a $j\geq m$ 
such that $\langle s_{m-1}\cdots s_{i+1}\alpha_i,\check \alpha_j\rangle\neq 0$.  
Take $i<m$. If $\langle s_{m-1}\cdots s_{i+1}\alpha_i,\check \alpha_m\rangle=0$, then $\langle s_{m-1}\cdots s_{i+1}\alpha_i,\check 
\alpha_h\rangle\neq 0$ for some $h>m$ and $\langle s_{m}\cdots s_{i+1}\alpha_i,\check \alpha_h\rangle=\langle s_{m-1}\cdots 
s_{i+1}\alpha_i,\check \alpha_h\rangle\neq 0$.

If $\langle s_{m-1}\cdots s_{i+1}\alpha_i,\check \alpha_m\rangle\neq 0$, denote by $\Gamma\subset \Delta_L$ the support of the root $ 
s_{m}\cdots s_{i+1}\alpha_i$. Remark that $\alpha_m\in \Gamma$, that $s_{m-1}\cdots s_{i+1}\alpha_i$ is supported on 
$\Gamma\setminus\{\alpha_m\}$ and that the support of $\Gamma$ is a connected subset of the Dynkin diagram. By assumption there is 
a $j>m$ such that $\langle \alpha_m,\check \alpha_j\rangle\neq 0$. Then $\langle s_{m-1}\cdots s_{i+1}\alpha_i,\check 
\alpha_j\rangle=0$. Otherwise there would be $\alpha_h\in \Gamma\setminus\{\alpha_m\}$ with $\langle\alpha_h,\check \alpha_j\rangle\neq 
0$ and the Dynkin diagram would contain a cycle. It follows that $$\langle s_{m}\cdots s_{i+1}\alpha_i,\check \alpha_j\rangle=\langle 
s_{m-1}\cdots s_{i+1}\alpha_i, s_m\check\alpha_j\rangle=$$$$-\langle \alpha_j,\check \alpha_m\rangle\langle s_{m-1}\cdots 
s_{i+1}\alpha_i,\check \alpha_m\rangle\neq 0$$ as desired. \end{proof}

We say that an order of the simple roots which satisfies the condition
of the previous Lemma is compatible with $L$.

\medskip

We explain now what we mean by a simultaneous quotient.
Let $B^-$ be the opposite Borel
subgroup of $B$, $U$ and $U^-$ the unipotent radicals of $B$ and $B^-$.
Let $\mathcal L=\{L_1\subset L_2\subset\cdots\subset L_h=G\}$ be a sequence of
standard (w.r.t.\ given choice of $T$ and $B$) Levi subgroups of $G$.
For each $i=1,\dots,h$ let $P_i=L_iB_i$ be the standard parabolic
subgroup with Levi factor equal to $L_i$. Let $V_i$ be the unipotent
radical of $P_i$ and let $\pi_i:P_i \lra L_i$ be the projection onto the
Levi factor. Similarly, considering the opposite parabolic $P_i^-$,
define $V_i^-$ and $\pi_i^-$. Let also $q^{ss}_i:L^{ss}_i \lra
L^{ss}_i/\!/Ad(L^{ss}_i)$ be the quotient map.  We consider the
following ``simultaneous'' quotient map
$$ q^{ss}_\calL: U \times U^- \lra \prod _{i=1}^h L^{ss}_i/\!/Ad(L^{ss}_i) $$
defined by
\begin{equation}\label{eq:barqss}
q^{ss}_\calL (u,v)=
\bigg(q^{ss}_1\big(\pi_1(u)\pi_1^-(v)^{-1}\big),\,q^{ss}_2\big(\pi_2(u)\pi_2^-(v)^{-1}\big),
\dots,\,q^{ss}_h\big(uv^{-1}\big)\bigg).
\end{equation}

In this generality the map $q^{ss}_\calL$ cannot have a section since it is
in general not surjective. To avoid this problem we give the following
definition.

\begin{definition}\label{dfn:admss}
A sequence $\mathcal L=\{L_1\subset L_2\subset\cdots\subset L_h=G\}$
of standard Levi subgroups is called ss-admissible if for all $i=1,\ldots ,h-1$,
each simple factor of $L_i$ does not coincide with a simple factor of
$L_{i+i}$.
\end{definition}

From now on we fix an ss-admissible sequence of Levi subgroups $\mathcal
L=\{L_1\subset L_2\subset\cdots\subset L_h=G\}$. Let
$\Phi_i=\Phi_{L_i}$ and $\Delta_i = \Delta_{L_i}$.  

For $i>1$, fix an order of $\Delta_i$ which is compatible
with $L_{i-1}$ and an arbitrary order of $\Delta_1=\Delta_{L_1}$. Let
$St_i : \mk^{\Delta_i} \lra L_i^{ss}$ be the Steinberg section defined
by this order.

\begin{remark}
Notice that it is not always possible, even changing the system of
simple roots, to choose an order of $\Delta$ which is compatible with
all the inclusions $L_s\subset L_{s+1}$ at the same time.
\end{remark}

Remark that if $i>1$, whatever compatible order we choose for
$\Delta_i$, the roots in $\Delta_{i-1}$ form the initial segment in
our order. Also notice that, for $i<h$, we get another order on
$\Delta_{i}$, the one obtained by restricting the order of
$\Delta_{i+1}$ to $\Delta_i$. Let $St'_i:\mk^{\Delta_i}\lra L_i^{ss}$
be the Steinberg sections defined using this second order.

For $i=1,\dots,n$ we define ``simultaneous'' Steinberg sections by
$$ St_{\mathcal L_i}:\prod_{i=1}^j\mk^{\Delta_i}\to L_i^{ss} \quad
\text{ by }\quad St_{\mathcal L_i}(\aundera^{(1)}, \ldots
\aundera^{(i)})=St_i\left(\sum_{j=1}^i\aundera^{(j)}\right) $$
and similarly for $i=1,\dots,h-1$
$$ St'_{\mathcal L_i}:\prod_{i=1}^j\mk^{\Delta_i}\to L_i^{ss} \quad
\text{ by }\quad St'_{\mathcal L_i}(\aundera^{(1)}, \ldots
\aundera^{(i)})=St'_i\left(\sum_{j=1}^i \aundera^{(j)}\right). $$
Let us remark that for $i=h$,
$$ St_{\mathcal L_h}( \aundera^{(1)}, \ldots, \aundera^{(h)})
= X_{G}\left(\sum_{j=1}^h \aundera^{(j)}\right)\cdot w $$

We need the following Lemma.

\begin{lemma} \label{piumenopiu}
For any ordering $\alpha_1,\ldots ,\alpha_n$, setting
$\grb_i=s_n\cdots s_{i+1}(\gra_i)$ for each $i=1,\ldots n$, the
element
$$\prod_{i=1}^nX_{\beta_i}(-1)s_n\cdots s_1$$ lies in $U^-U^+$. In
particular the element $s_n\cdots s_1$ can be written as a product
$u_+u_-v_+$ where $u_+,v_+\in U_+$ and $u_-\in U_-$.
\end{lemma}

\begin{proof} Set $\Theta=\prod_{i=1}^nX_{\beta_i}(-1)s_n\cdots s_1$.

Let $\omega_i$  denote the fundamental weight
such that $\langle\omega_i,\check \alpha_j\rangle=\delta_{i,j}.$ In the fundamental
representation $V_{\omega_i}$ of highest weight $\omega_i$, choose
 a highest weight vector $v_i$.   We denote
by $W_i\subset V_{\omega_i}$, the $T$ invariant complement to the one
dimensional space spanned by $v_i$.

In order to prove our claim, we need to see that    $\Theta v_i=v_i+w_i$ with
$w_i\in W_i$.

Write $$\prod_{i=1}^nX_{\beta_i}(-1)s_n\cdots s_1=X_{\alpha_n}(-1)s_n\cdots
X_{\alpha_1}(-1)s_1$$ and remark that $W_i$ is stable under
$X_{\alpha_j}(-1)s_j$ for each $j\neq i$, and that furthermore
$X_{\alpha_j}(-1)s_jv_i=v_i.$ On the other hand, a direct computation
shows that $X_{\alpha_i}(-1)s_iv_i=v_i+w'_i$ with $w'_i\in W_i$. It
follows immediately that $\Theta v_i=v_i+w_i$ with $w_i\in W_i$. This is
our claim.
\end{proof}

\begin{notation}\label{not:est}
In what follows we will use the following notation.

If
$i_1<\cdots<i_j$ and $f: \prod_{\ell=1}^j\mk^{\Delta_{i_\ell}} \lra
X$ is a function with values in a set $X$, then, for all
$\undera=(\aundera^{(1)},\dots,\aundera^{(h)})\in
\prod_{i=1}^h\mk^{\Delta_i}$ by $f(\undera)$ we will denote
$f(\aundera^{(i_1)},\dots,\aundera^{(i_j)})$.

We will use this notation also if we have an action of a group on
$\prod_{\ell=1}^j\mk^{\Delta_{i_\ell}}$.
\end{notation}

\begin{lemma}\label{lemma:VZW}
Given an ss-admissible sequence $\mathcal L$, there exists morphisms
$$V: \prod_{i=1}^{h-1}\mk^{\Delta_i} \to V_{L_{h-1}}, \qquad Z:
  \mk^{\Delta_h}\to V_{L_{h-1}}, \qquad W:
  \prod_{i=1}^{h-1}\mk^{\Delta_i}\to V^-_{L_{h-1}}$$ such that
\begin{equation}\label{labella}
V(\undera)\,t\,St_{\mathcal
L_h}(\undera)V^{-1}(\undera)=V(\undera)\,t\,Z(\undera)\,St'_{\mathcal
L_{h-1}}(\undera)\,W(\undera)
\end{equation}
for all $\undera \in \prod_{i=1}^h \mk^{\Delta_i}$ and for all $t\in T$.
\end{lemma}

\begin{proof}
Let $\gra_1,\dots,\gra_n$ be the order we have fixed for
$\Delta=\Delta_h$ and let    $m=|\Delta_{h-1}|$. Set   $s_i=s_{\gra_i}$ and define $\grb_i$ as in
Lemma \ref{lemma:compatibile}. Using Lemma \ref{lemma:compatibile}, we
get that the roots $\beta_1,\ldots \beta_n$ do not lie in
$\Phi_{L_{h-1}}$. Hence the image of $X_G$ is contained in $V_{L_{h-1}}$.

Let us set $w'=s_n\cdots s_{m+1}$. Then we clearly have
$$ t\,St_{\mathcal L}(\aundera^{(1)}, \ldots ,\aundera^{(h)})=
t\,X_G(\aundera^{(h)})w'St'_{\mathcal L_{h-1}}(\aundera^{(1)}, \ldots
\aundera^{(h-1)}).$$ Now write, using Lemma \ref{piumenopiu},
$w'=u_+u_-v_+$. Remark that all the elements involved are in the Levi
subgroup associated to the simple roots
$\Delta_h\setminus\Delta_{h-1}$ and, as a consequence, $u_+,v_+\in
V_{L_{h-1}}$ and $u_-\in V^-_{L_{h-1}}$. Set $V =
(St'_{\calL_{h-1}})^{-1}v_+ St_{\calL_{h-1}}$ and $Z:=X_G\,u_+$.
Notice that both $V$ and $Z$ take their values in $V_{L_{h-1}}$.
Finally set $W = (St'_{\mathcal L_{h-1}})^{-1}u_- St'_{\mathcal
  L_{h-1}}$ and notice that it takes values in $V^-_{L_h}$. Identity
\eqref{labella} now follows.
\end{proof}

The previous Lemma can be used inductively.

\begin{lemma}\label{lemma:Stss}
For every ss-admissible sequence of Levi factors there exist:
\begin{enumerate}[a)]\item morphisms
$$
A:T\times\prod_{i=1}^{h}   \mk^{\Delta_i} \lra U,  \qquad
C:T\times\prod_{i=1}^{h-1} \mk^{\Delta_i} \lra U^- ,$$\item
morphisms
$$U_i:\prod_{j=1}^{i-1}      \mk^{\Delta_j} \lra L_i,
$$ for   $i=1,\dots,h$, \item
  elements $w_i$ in the Weyl group $N_{L_i}(T)/T$, \item  actions $t \odot_{i} \undera$ of $T$ on
$\mk^{\Delta_i}$, for $i=1,\dots,h$, having the property that $\odot_{h}$ is trivial and  $\odot_{i}$ leaves every line
$\mk^{\gra}$ stable, \end{enumerate}such that
\begin{equation}\label{eq:AC}
t \, \pi_i\big(A(t,\undera)\big) \, \pi_i^-\big(C(t,\undera)\big) =
U_i(t,\undera) w_i(t) St_{\calL_i}(t \odot_{i} \undera)
U_i(t,\undera)^{-1}
\end{equation}
for every $t\in T$ and $\undera \in \prod_{i=1}^{h} \mk^{\Delta_i}$
(we are using Notation \ref{not:est}).
\end{lemma}

\begin{proof}
We proceed by induction on $h$.  If $h=1$, everything follows from
the previous Lemma setting $A(t,\undera)=t^{-1}V(0)tZ(\undera)$ and
$C(t,\undera)=W(0)$, $U_1(t,\undera)=V(0)$, $w_1=1$ and $\odot_1$
trivial.

If $h>1$, we can   assume that our statement holds for
$\mathcal L' =\{ L_1\subset \dots \subset L_{h-1}\},$ so there exist
morphisms $U_i',A',B'$ elements $w_i'$ and actions $\odot_{i}'$
 satisfying our requirements. We choose morphisms $V,Z,W$ as in the
previous Lemma \ref{lemma:VZW}.

By Lemma \ref{lemma:riordinared} there exist a
$g_0:\mk^{\Delta_{h-1}}\lra L^{ss}_{h-1}$, an element of the Weyl
group $w$ and an action $\cdot$ of $T$ on $\mk^{\Delta_h}$ which
preserves every line $\mk^\gra$, such that $$t\,St_{h-1}'(b )
= g_0(t\cdot b) \, w(t)\, St_{h-1}(t\cdot b) \, g_0^{-1}(t\cdot
b)^{-1},$$ for all $t\in T$ and $b \in \mk^{\Delta_h}$. Extend $g_0$
to a map $g:  \prod_{i=1}^{h-1}\mk^{\Delta_i} \lra L^{ss}_{h-1}$ by
$$g(a^{(1)},\dots,a^{(h-1)})= g_0\left(\cdot\sum_{i=1}^{h-1}
a^{(i)}\right).$$ By definition we have $$ t\,St'_{\calL_{h-1}}(\undera) =
g(t\cdot\undera) \,w(t)\, St_{\calL_{h-1}}(\undera)\,
g(t\cdot\undera)^{-1}.$$

For all $t\in T$, $\underb \in \prod_{i=1}^{h-1}\mk^{\Delta_i}$ and
all $\undera \in \prod_{i=1}^{h}\mk^{\Delta_i}$ define
\begin{align*}
E_h(t,\underb) &=  U'_{h-1}\big(t,w(t)^{-1} \cdot \underb
\big) \, g\big(w(t)^{-1}\cdot \underb\big)^{-1} \,V(\underb) \\
A_h(t,\undera) &= t^{-1} \, E_h(t,\undera) \, V(\undera) \, t \,
Z(\undera) \, t^{-1} \, E_h(t,\undera)^{-1}\, t \\
C_h(t,\underb) &= E_h(t,\underb) \, W(\underb) \,
E_h(t,\underb)^{-1} \\
U_h(t,\underb) &=E_h(t,\underb) \, V(\underb)
\end{align*}
and notice that $E_h$ takes values in $L_{h-1}$, $A_h$ in $V_{h-1}$,
$B_h$ in $V_{h-1}^-$ and $U_h$ in $L_h$. By \eqref{labella} we have
$$
U_h(t,\undera) \, w(t)^{-1} \, St_{\calL_h}(\undera) \,
U_h(t,\undera)^{-1}= $$$$= t\, A_h(t,\undera)\,
A'\big(t,w(t)^{-1}\cdot\undera\big)\,
C'\big(t,w(t)^{-1}\cdot\undera\big)\, C_h(t,\undera).
$$
Now, for all $t\in T$, $\underb \in
\prod_{i=1}^{h-1}\mk^{\Delta_i}$ and all $\undera \in
\prod_{i=1}^{h}\mk^{\Delta_i}$, we define
\begin{align*}
A(t,\undera) &= A_h(t,\undera)\, A'\big(t,w(t)^{-1}\cdot\undera\big)
\\
C(t,\underb)&=C'\big(t,w(t)^{-1}\cdot\underb\big)\, C_h(t,\underb).
\end{align*}
Then $\pi_{h-1}\big(A(t,\undera)\big)=
A'\big(t,w(t)^{-1}\cdot\undera\big)$ and
$\pi_{h-1}^-\big(C(t,\underb)\big)
=C'\big(t,w(t)^{-1}\cdot\underb\big)$. Let $w_h = w^{-1}$,
$w_i=w_i'$. Take  $\odot_h$ to be the  trivial action and, for all $i=1,\dots,h-1$, $t \in
T$, $\underc \in\prod_{^j=1}^i\mk^{\Delta_i}$, set $t\odot_i
\underc := t\odot'_i (w(t)^{-1}\cdot \underc)$. Finally for all
$i=1,\dots,h-1$, $t \in T$, $\underc
\in\prod_{^j=1}^i\mk^{\Delta_i}$ and $t\in T$, define
$$
U_i(t,\underc) = U_i'\big(t,w(t)^{-1}\cdot \underc\big).
$$
By the inductive hypothesis and a straightforward computation our claim
 follows.
\end{proof}

A special case of the previous Lemma gives a section to $q^{ss}_\calL$.

\begin{proposition}\label{prp:Stss}
For every ss-admissible sequence of Levi factors $\calL$, there
exists a morphism $ \chi^{ss} : \prod_{i=1}^h \mk^{\Delta_i}\to U\times U^-$ such
that $q_\calL^{ss}\circ \chi$ is an isomorphism.
\end{proposition}

\begin{proof}Let $A, C$ be as in the previous Lemma and define
$$
\chi^{ss}(\undera) = \big(A(1,\undera),\,C(1,\undera)\big)
$$
By equation \eqref{eq:AC} for
$\big(\undera^{(1)},\dots,\undera^{(h)}\big)\in \prod_{i=1}^h
\mk^{\Delta_i}$, we have
\begin{align*}
q^{ss}_\calL \circ &\chi \big(\undera^{(1)},\dots,\undera^{(h)}\big) = \\
&\bigg( q^{ss}_1\big(St_1(\undera^{(1)})\big),
q^{ss}_2\big(St_2(\undera^{(1)}+\undera^{(2)})\big),\dots,
q^{ss}_h\big(St_h(\sum_{i=1}^{h}\undera^{(i)})\big) \bigg).
\end{align*}
which, by the properties of the Steinberg map, is clearly an
isomorphism.
\end{proof}

\subsection{General simultaneous Steinberg section}\label{ssec:simultanea}
We extend now our section taking into account also the fact that the
various $L_i$ are not semisimple. Let us start with our ss-admissible
sequence $\mathcal L=\{L_0\subset \ldots \subset L_h=G\}$ of standard
Levi subgroups. Let $q_i : L_i \lra L_i /\!/ Ad(L_i)$ be the quotient
map.  We can define $q_\calL : B \times U^- \lra \prod_{i=1}^h L_i/\!/
Ad(L_i)$ as in formula \eqref{eq:barqss}. However in general this map
is not surjective since the Levi subgroups may have common factors in
the center.

\begin{definition}\label{dfn:adm}
Fix a ss-admissible sequence of Levi factors $\mathcal L=\{L_0\subset
\ldots \subset L_h=G\}$. A sequence $\calM =\{ M_1\subset \dots
\subset M_h\}$ of connected subgroups of $G$ is said to be compatible
with $\calL$ if for all $i$, $M_i$ is a subgroup of $L_i$ containing
$L_i^{ss}$. 

We say that the compatible sequence $\calM$ is admissible, if there
exists subtori $S_i \subset T$ such that
 \begin{enumerate}[\indent {\bf S}1: ]
   \item $M_i = S_i \ltimes L_i^{ss}$;
   \item $S_{i+1} \cap M_i = \{1\}$.
 \end{enumerate}
In this case, notice that $ (S_i \times \dots \times S_h)\ltimes L_i^{ss}
\subset G$ is a semi-direct product.
If the sequence $\calM =\{M\subset G\}$ is admissible, we say that $M$
is admissible.
\end{definition}

In what follows we are going to need a number of simple general
remarks which we collect in the following:

\begin{lemma}
\begin{enumerate}[i)]
\item If $S\subset T$ are two tori, there exists a subtorus $R$ of $T$
  such that $T=R\times S$. We call $R$ a complement of $S$ in $T$.
\item Let $M\subset N$ be two reductive connected groups with the same
  semisimple part. If $T$ is a maximal torus of $N$, $T\cap M$ is a
  maximal torus of $M$.
\item Let $D$ be a lattice. Let $A,B,C \subset D$ be sub-lattices such
  that $A$ is saturated in $D$, $A\cap B =C\cap B= \{0\}$ and $A+B$
  has finite index in $D$. Then there exists a sub-lattice $B'$ of $D$
  such that $C\cap B'= \{0\}$ and $A\oplus B'=D$.
\item If $M\subset N$ are connected reductive groups with $|Z(N)\cap M|<\infty$ and  $T$ a maximal
  torus of $N$ which intersects $M$ in a maximal torus of $M$,  there exists a subtorus $S\subset T$ such that
  $N=S\ltimes N^{ss}$ and $S\cap M= \{1\}$.
\end{enumerate}
\end{lemma}
\begin{proof}
We have already proved $i)$ in section \ref{ssec:Steinbergred}. To
prove $ii)$ notice that $T\cap M$ contains a maximal torus of $M$ and
that it is commutative. Since in a connected group, the centralizer of
a maximal torus is always connected $T\cap M$ is a maximal torus of
$M$.

To prove $iii)$, let $a,b,c,d$ be the ranks of $A,B,C,D$
respectively. By our hypothesis we have $a+b=d$ and $c\leq a$. We can
choose a basis $e_1,\dots,e_a$ of $A$ and extend it to a basis
$e_{a+1},\dots,e_{d}$ of $D$. Let $v_1,\dots,v_c$ be a basis of
$C$. For $u_1,\dots,u_b\in A$ consider the span $B'(u_1,\dots,u_b)$ of
$\{w_i=e_{a+i}+u_i \st i=1,\dots,b\}$. This is a complement of $A$ in
$D$.  $B'(u_1,\dots,u_b)$ intersects $C$ if and only if all the
maximal   minors of the matrix whose columns are given by
$w_1,\dots,w_b, v_1,\dots,v_c$ vanish.

Thus, if for each choice of $u_1,\dots,u_b$ the corresponding
$B'(u_1,\dots,u_b)$ intersects $C$, it means that these minors define
polynomial functions on $A^{b}$ which are identically zero.  However,
if we tensor with the rational numbers, the existence of $B$
guarantees that there exist vectors $u_1,\dots,u_b\in A \otimes_\mZ
\mQ$ on which the value of these polynomial functions is non zero. By
the density of the integers in the Zariski topology we get a
contradiction.

Point $iv)$ is a consequence of $iii)$. Recall that if $R$ is a torus
with lattice of cocharacters $\Lambda_*(R)$, then
$R=\mk^*\otimes_{\Bbb Z}\Lambda_*(R) $. Now, take $D$ to be the lattice of
cocharacters of $T$, $A$ to be the set of cocharacters of $T\cap
N^{ss}$, $B$ the set of cocharacters of the identity component of
$Z(N)$ and $C$ the set of cocharacters of $T\cap M$. Choose $B'$ as in
$iii)$ and set $S=\mk^*\otimes_{\Bbb Z} B'$.
\end{proof}

 A maximal admissible  sequence compatible with $\calL$ can be constructed in the
following way: choose $S_h$ to be a complement torus of $L^{ss}_h$ in
$L_h$ and, for every $i=1,\ldots ,h-1$, $S_i$ to be a complement of $L_{i}^{ss}$ in
$L_{i+1}^{ss}$. Then define $M_i=S_i \ltimes L_i^{ss}$ and notice that
$\calM$ is an admissible sequence and that $L_i$ is the semi-direct
product $S_i\times \dots \times S_h \ltimes L_i^{ss}$ for
$i=1,\dots,h$.

Another class of admissible sequences is described in the following
Lemma.

\begin{lemma}Let $\calM =\{ M_1\subset \dots \subset M_h\}$ be  a sequence
compatible with a ss-admissible sequence $\calL$. Suppose that for
all $i=1,\dots,h-1$ the intersection $Z(M_{i+1})\cap M_i$ is finite.
Then $\calM$ is admissible.
\end{lemma}

\begin{proof}
By property $iv)$ above we can construct $S_{i}$ such that $S_i\cap
M_{i-1}= \{1\}$ and $M_{i}=S_i \ltimes L_i^{ss}$.
\end{proof}

 We now fix an admissible sequence $\calM$ and subgroups $S_i$ with
 properties S1 and S2. Let $S_i'=S_i \times \dots \times S_h$.  For
 each $i$ let $R_i \subset T$ be such that $L_i$ is the semi-direct
 product $R_i \times S'_i \ltimes L_i^{ss}$.
 Let $p_i:L_i \lra L_i^{ss}/\!/Ad(L_i^{ss})$ be an extension to $L_i$
 of the   quotient map defined on $L_i^{ss}$ as in Section \ref{ssec:Steinbergred}.
 Then we can identify the adjoint quotients $q^M_i:M_i\lra
 M_i/\!/Ad(M_i) $ and $q_i:L_i\lra L_i/\!/Ad(L_i)$ with the maps $S_i
 \ltimes L_i^{ss} \lra S_i \times L^{ss}_i/\!/Ad(L^{ss}_i) $ given by
 $(s,g)\mapsto (s,p_i(s\cdot g))$ and $R_i \times S'_i \ltimes
 L_i^{ss} \lra R_i \times S'_i \times L^{ss}_i/\!/Ad(L^{ss}_i)$ given
 by $(r,s,g)\mapsto (r,s,p_i(r\cdot s \cdot g))$.  Considering the
 composition of $q_i$ with the projection on the factor $S_i$ of $S_i'$
 we get a projection $r_i: L_i \lra S_i \times
 L^{ss}_i/\!/Ad(L^{ss}_i) $ and under  the above identifications   we obtain $q_i^M(m)=r_i(m)$
 for each $m
 \in M_i$.

 With this notations we can define a ``simultaneous'' quotient map
 $$
 r_\calM : B \times U^- \lra \prod_{i=1}^h M_i/\!/ Ad(M_i)
 $$
 defined by
 $$
 r_\calM (u,v) = \bigg(r_1\big(\pi_1(u)\pi_1^-(v)\big),\,r_2\big(\pi_2(u)\pi_2^-(v)\big),
 \dots,\,r_h\big(uv \big)\bigg).
 $$
 We can now state our Theorem on simultaneous section.

 \begin{theorem}\label{teo:simultanea}
 Let $\calM$ be an admissible sequence as in definition \ref{dfn:adm}
 and fix subgroups $S_i$, $R_i$ as above so that the map $r_\calM$ is
 defined. Let $S=S_1\times \cdots S_h$. Then there exists $\chi: S
 \times \prod_{i=1}^h \mk^{\Delta_i} \lra B\times U^-$ such that
 $r_\calM \circ \chi :S \times \prod_{i=1}^h \mk^{\Delta_i} \lra
 \prod_{i=1}^h M_i/\!/Ad(M_i)$ is an isomorphism.
 \end{theorem}

 \begin{proof}
 Let $A, C$ be as in Lemma \ref{lemma:Stss} and define
$$
 \chi(s,\undera) = \big( s\, A(s,\undera) ,\, C(s,\undera)\big).
$$
 We prove that $r_\calM \circ \chi$ is an isomorphism.
 
 Let $s=t_1\cdots t_h$ with $t_i \in S_i$. By property S1, for each
 $i=2,\dots,h$, there exist functions $\lambda_i:S_1\times \dots
 \times S_{i-1} \lra L_i^{ss}$ and $\mu_i:S_1\times \dots \times
 S_{i-1} \lra S_i$, such that
 $$
 t_1\cdots t_{i-1}=\lambda_i(t_1,\dots, t_{i-1}) \mu_i(t_1,\dots,
 t_{i-1}).
 $$
 Then with the notation of Lemma \ref{lemma:Stss} we have
 \begin{align*}
 r_i\big( \chi(s,\undera) \big) & = r_i\bigg( s \, \pi_i \big(
 A(s,\undera) \big) \, \pi_i^{-}\big(C(s,\undera)\big)\bigg) \\
 &= \bigg( t_i \mu_i(t_1\cdots t_{i-1}) , p_i\big(s \, \pi_i \big(
 A(s,\undera) \big) \, \pi_i^{-}\big(C(s,\undera)\big)\bigg) \\
 &= \bigg( t_i \mu_i(t_1\cdots t_{i-1}) , p_i\big(w_i(s)
 St_{\calL_i}(s\odot_i \undera) \big)\bigg).
 \end{align*}
 At this point everything follows easily from \ref{prp:Steinbergred}.
 \end{proof}

We can consider also the following slightly different ``simultaneous''
quotient map that will be needed later on. Let
$S=\prod_{i=1}^h S_i$ be as in Theorem \ref{teo:simultanea}. Define
$q^{ss}_\calM : SU \times U^- \lra \prod_{i=1}^h L^{ss}_i/\!/
Ad(L^{ss}_i)$ by formula \eqref{eq:barqss}, and finally define
$$
\bar r_\calM : SU \times U^- \lra S \times \prod_{i=1}^h L^{ss}_i/\!/ Ad(L^{ss}_i)
$$ by $\bar r_\calM (su,v)= \big(s, q^{ss}_\calM(s u,v) \big)$ for all
$s\in S$, $u\in U$ and $v\in U^-$.
A variant of Theorem  \ref{teo:simultanea} which we are going to use later, is given by

\begin{lemma}\label{lem:simultanea}
Let $\calM$ be an admissible sequence and let $S,\bar r_\calM$ as
above.  Then there exists $\chi: S \times \prod_{i=1}^h \mk^{\Delta_i} \lra SU\times
U^-$ such that $\bar r_\calM \circ \chi :S \times \prod_{i=1}^h
\mk^{\Delta_i} \lra S\times \prod_{i=1}^h L^{ss}_i/\!/Ad(L^{ss}_i)$ is an
isomorphism. Moreover the $S$ component of $\bar r_\calM(\chi(s,x))$ is equal to $s$.
\end{lemma}
\begin{proof}
The same function $\chi$ defined in the proof of Theorem \ref{teo:simultanea}
satisfies the requirements of the Lemma. The proof is completely analogous.
\end{proof}

 \section{The $GL(n)$ case}
 The proof of Theorem \ref{teo:simultanea} is constructive.  However
 it is difficult to write down an explicit general formula. In this
 section we construct a very explicit section in the case of $GL(n)$
 which seems to us to be particularly nice.

 We start by giving a curious parametrization of the Borel subgroup
 of lower triangular matrices.

 Let $A=(a_{i,j})$ be a lower triangular $n\times n$ matrix and let
 $C=(c_{i,j})$ be the upper $n\times n$ triangular matrix with $c_{i,j}=1$ for
 all $i\leq j$.

 Take the product matrix $D=AC$ and, for any $1\leq i_1<i_2<
 i_r\leq n$, denote by $[i_1,i_2,\ldots ,i_r]$ the determinant of the
 principal minor of $D$   consisting of the rows (and columns) of
 index $i_1,i_2,\ldots ,i_r$.

 \begin{proposition} We have, setting $i_0=0$,   
 \begin{equation}\label{indu}
 [i_1,i_2,\ldots ,i_r]=\prod_{h=1}^{r}(a_{i_h,i_h}+\cdots +a_{i_h,i_{h-1}+1}).\end{equation}
 \end{proposition}
 \begin{proof}
 Remark that, if $d_{h,k}$ is the entry of $D$,  in the $h$-th row and $k$-th column,
 \begin{equation}\label{espr}
 d_{h,k}=\begin{cases}  a_{h,1}+\ldots +a_{h,k}\text{\ if \ } k\leq h\\
 a_{h,1}+\ldots +a_{h,h}\text{\ if \ } k\geq h\end{cases}.
 \end{equation}
 From this our result is clear for $r=1$ and we can proceed by
 induction.

 Set $\underline i:=(i_1,\ldots ,i_r)$ and denote by
 $D(\underline i)$ the corresponding principal minor. The last two columns of
 $D(\underline i)$ are
 $$\left(\begin{array}{ll}\sum_{s=1}^{i_1}a_{i_1,s} &
   \sum_{s=1}^{i_1}a_{i_1,s} \\\sum_{s=1}^{i_2}a_{i_2,s} &
   \sum_{s=1}^{i_2}a_{i_2,s} \\....... &
   ....... \\\sum_{s=1}^{i_{r-1}}a_{i_{r-1},s} &
   \sum_{s=1}^{i_{r-1}}a_{i_{r-1},s} \\\sum_{s=1}^{i_{r-1}}a_{i_r,s}
   &\sum_{s=1}^{i_{r}}a_{i_r,s}\end{array}\right)$$
 Substituting the last column with the difference of the last two columns, we
 deduce
 $$[i_1,i_2,\ldots ,i_r]=[i_1,i_2,\ldots
   ,i_{r-1}](a_{i_r,i_r}+\cdots +a_{i_r,i_{r-1}+1}).$$
 From this everything follows.
 \end{proof}

 This Proposition has some simple consequences.

 \begin{proposition}\label{poli}
 Let $P_r$ be the $r$-th coefficient of the characteristic polynomial of $D$. Then
 \begin{enumerate}
 \item $P_r$ does not depend from $a_{h,k}$ if $n-h+k< r$.
 \item $P_r$ depends linearly from $a_{h,k}$ if $n-h +k= r$. Furthermore, if  $n-h +k= r$,
 the coefficient of $a_{h,k}$ is  $\prod_{t=1}^{k-1}a_{t,t}\prod_{t=h+1}^{n}a_{t,t}$.
 \end{enumerate}
 \end{proposition}
 \begin{proof}
 Denote by  $\Lambda_r$ the set of subsets of  $\{1,\ldots ,n\}$ of
 cardinality $r$. Any such subset has a obvious total order. We have
 $$P_r=\sum_{\{i_1,\ldots ,i_r\}\in \Lambda_r}[i_1,i_2,\ldots ,i_r].$$
 As above, set $i_0=0$ and observe that, given $\underline
 i:=(i_1,\ldots ,i_r)$, for each $s$, $i_s-i_{s-1}\leq n-r+1$. If
 $a_{h,k}$ appears in $[i_1,i_2,\ldots ,i_r]$, then
 necessarily there is a $1\leq s\leq r$ with $i_s=h$. Furthermore we
 must also have $h-k\leq h-i_{s-1}-1\leq n-r$ and so $n-h+k\geq r$.

 This proves our first claim.

 As for our second claim, notice that if $n-h+k=r$ and $i_s=h$ then
 necessarily, since $i_s-i_{s-1}=n-r+1$, $i_{s-1}=k$. It follows   that
 $[i_1,i_2,\ldots ,i_r]=[1,2,\ldots ,k-1,h, h+1,\ldots ,n]$. But, by Proposition \ref{poli},
 $$[1,2,\ldots ,k-1,h, h+1,\ldots
   ,n]=\sum_{k}^ha_{h,j}\prod_{t=1}^{k-1}a_{t,t}\prod_{t=h+1}^{n}a_{t,t}$$
 and everything follows.
 \end{proof}

 Given an $n\times n$ matrix $X$, set   $X_i$ equal  to the $i\times i$ matrix
 obtained from $X$ erasing the last $n-i$ rows and columns. Notice that 
  $D_i=A_iC_i$.

 If we let $A$ vary in the Borel subgroup $B\subset GL(n)$ of lower
 triangular matrices, we obtain, for each $i=1,\ldots ,n$, a map
 $\phi_i: B\to GL(i)$ defined by $\phi_i(A)=A_i$. Composing with the
 map associating to each matrix in $GL(i)$ the coefficients
 $P_1^{(i)},\ldots , P_i^{(i)}$ of its characteristic polynomial (we
 assume that $P_h^{(i)}$ has degree $h$, so that in particular
 $P_i^{(i)}$, being the determinant, takes non zero values on $GL(i)$),
 we get a morphism
 $$c_i: B\to \mk^{i-1}\times \mk^*$$
 defined by
 $c_i(A)=(P_1^{(i)}(D_i),\ldots , P_i^{(i)}(D_i))$ .
 If we take the map $\Pi=\times_{i=1}^nc_i$, we get a map
 $$\Pi:B\to  \mk^{\frac{n(n-1)}{2}}\times (\mk^*)^n.$$

 \begin{theorem}\label{teo:GL} The map $\Pi$ is an isomorphism. \end{theorem}

 \begin{proof}
 We are going to explain how to construct an inverse to $\Pi$. This
 will follow if we show that once we have fixed the values $z_h^{(k)}$ of the
 functions $P_h^{(k)}$, $1\leq k\leq h\leq n$, there exists a unique
 $A=(a_{i,j})$ in $B$ such that $z_h^{(k)}=P_h^{(k)}(D_k)$.

Let us start with the diagonal entries.  We must have, for $i=1,\ldots ,n$,
 $$z_i^{(i)}=\prod_{h=1}^ia_{h,h}.$$ Since $z_h^{(h)}\in \mk^*$ we get
 $$ a_{h,h}=\frac {z_h^{(h)}}{z_{h-1}^{(h-1)}}.$$

 Let us now do induction on $h-k=r$. Assume that the entries $a_{p,q}$
 with $p-q<r$ can be uniquely determined by  
 $z_s^{(t)}$ with $s-t<r$.  By Proposition \ref{poli} we deduce,
 $$(\prod_{t=1}^{h-1}a_{t,t})^{-1}z^{(h+r)}_{h}=\frac{z^{(h+r)}_{h}}{z^{(h-1)}_{h-1}}=a_{h+r,h}+F_{h}(r)$$
 where $F_{h}(r)$ is a degree one polynomial in the entries $a_{m+r,m}$
 with $m<h$ whose coefficients are    polynomials in the entries $a_{s,s}^{\pm}$,
 $a_{s,t}$ with $0\leq s-t< r$. Since, by induction all the $a_{s,t}$ with $0\leq s-t<r$ have
 been already determined, 
 the entries $a_{h+r,h}$ are solutions of a linear system in triangular
 form with $1$'s on the diagonal and hence are uniquely determined.
 \end{proof}

 We consider now the Gelfand-Zeitlin  sequence
 $$
\{ GL(1)\subset \dots \subset GL(n)\}
 $$
 where $GL(i)$ is the subgroup leaving invariant the last $n-i$
 coordinates.  Notice that this is an admissible sequence. The map
 $r=r_\calM:B \times U^- \lra \prod_{i=1}^n GL(i)/\!/Ad(GL(i))$
 of the previous section can be described in the following way
 $$
 r(U,V) = \big( P_j^{(i)}(U_iV_i)\big)\ \ \text{for\ \ }1\leq j\leq i\leq n
 $$
  Hence Theorem
 \ref{teo:GL} gives an explicit form of Theorem \ref{teo:simultanea}:

 \begin{corollary}
 Let $\chi: B \lra B \times U^-$ be defined by $\chi(A)= (A,C)$.
 Then $r\circ \chi$ is an isomorphism.
 \end{corollary}

 Notice that a similar result holds if we take a subsequence of the sequence above:
 $$
 GL(n_1)\subset GL(n_2) \subset \dots \subset GL(n_h)
 $$ with $n_1< \dots <n_h=n$. In this case define $\calA$ as the
 subset of $B$ of matrices $A=(a_{i,j})$ such that for $m\neq n_h$, $a_{m,j}=\delta_{m,j}$. Then, if we define $\chi:\calA \lra B \times U^-$ by
 $\chi(A)=(A,C)$, the same proof shows that $r \circ \chi$ is an
 isomorphism.

 \section{Poisson commutative subalgebras of the algebra of the Poisson dual of $G$}
\label{sez:app1}
 We want to apply our result on the existence of a generalized
 Steinberg section to the study of some Poisson algebras arising from
 Manin triples.  Recall that a Manin triple is a triple
 $(\gog,\goh,\gok)$ where $\gog$ is a Lie algebra equipped with a non degenerate
 invariant bilinear form $\kappa$, $\goh$, $\gok$ are Lie subalgebras
 of $\gog$ which are maximal isotropic subspaces with respect to
 $\kappa$ and such that $\kappa$ induces a perfect pairing between
 $\goh$ and $\gok$, so that $\gog=\gok\oplus \goh$.  If $H$ is a
 connected group with Lie algebra equal to $\goh$, considering left
 invariant vector fields, we identify the tangent bundle on $H$ with $H
 \times \goh$, the cotangent bundle on $H$ with $H \times \gok$ and
 if $f$ is a function on $H$, we denote with $\grd_x f\in \gok$ the
 differential of $f$ at $x$ w.r.t.\ this isomorphism. Assume now that
 $H$ is a subgroup of a group $G$ with Lie algebra equal to $\gog$ so
 that $H$ acts on $\gog$ preserving the form $\kappa$.  A Poisson
 structure on $H$ is then defined in the following way
 $$ \{ f , g \} (x) = \kappa (\grd_xf ; Ad_x(\grd_xg)) -\kappa (\grd_xg
 ; Ad_x(\grd_xf))$$ for all $x\in H$ and $f,g$ functions on
 $H$.
 If $(\gog_1,\goh_1,\gok_1)$ and $(\gog_2,\goh_2,\gok_2)$ are two Manin
 triples   and $\grf : \gog_1 \lra \gog_2$ is a morphism of Lie
 algebras such that $\grf(\goh_1)\subset \goh_2$ and $\grf(\gok_1)
 \subset \gok_2$ and $\phi : H_1 \lra H_2$ is a group homomorphism   
   whose differential is equal to $\grf$, then $\phi$
 does not need to be a Poisson map. However we have the following
 Lemma.

 \begin{lemma}\label{lem:Manintriple}
Let $(\gog_1,\goh_1,\gok_1)$ and $(\gog_2,\goh_2,\gok_2)$ be  two Manin
 triples. Let  
 $\kappa_i$ be the invariant bilinear form on $\gog_i$, $i=1,2$. Let $\grf :
 \gog_1 \lra \gog_2$ be such that
 \begin{enumerate}[\indent i)]
    \item $\grf$ is a morphism of Lie algebras;
    \item $\grf(\gok_1)\subset \gok_2$, and $\grf(\goh_1)\subset \goh_2$;
    \item $\kappa_2(\grf(u),\grf(v))=\kappa_1(u,v)$ for all $u,v \in \gog$;
    \item $\psi=\grf^*:\frak h_2 \lra \frak h_1$ is a morphism of Lie algebras.

 \end{enumerate}
 For $i=1,2$, let $G_i$ be a group with Lie algebra equal to $\gog_i$
 and let $H_i\subset G_i$ be a connected subgroup with Lie algebra
 equal to $\goh_i$ and consider on $H_i$ the Poisson structure
 introduced above. Let  $\Phi:H_1\lra H_2$ and $\Psi:H_2\lra
 H_1$ be  homomorphisms whose differentials are  $\grf$ and $\psi$.

 Then the map $\Psi$ is a morphism of Poisson groups.
 \end{lemma}

 \begin{proof}
Notice first that, since the bilinear form $\kappa_1$ is non
 degenerate, we have $\psi \circ \grf = id$, hence $\Psi
 \circ \Phi = id$. So, if $N = \ker \Psi$ we have $H_2 \isocan N
 \rtimes H_1$, in particular $N$ is connected. Now we prove that for
 all $u,v \in \gok_1$ and for all $x\in H_2$
 $$\kappa_2(\grf(u),Ad_x(\grf(v))) = \kappa_1(u,Ad_{\Psi(x)} v).$$

 Indeed if $x= \Phi(y)$,  since $\Psi\circ\Phi=id$ , this is clear by 
 property $iii)$. We prove now the statement for $x \in N$. In this case
 we need to prove that $\kappa_2(\grf(u),Ad_x(\grf(v))) = 0$.  Let
 $\gon$ be the Lie algebra of $N$. Consider the subspace $V=\gon \oplus
 \grf(\gok_1)$.  Notice that $V$   is maximal isotropic
 with respect to \ $\kappa_2$.  Moreover, a simple computation shows that it is
 stable under the action of $\gon$, so it is stable also under the
 action of $N$. In particular for $x \in N$ and $u,v \in \gok_1$ we
 have $\kappa_2\big(\grf(u),Ad_x(\grf(v))\big) = 0$ as desired.

 Let now $f,g$ be two functions on $H_1$, $x\in H_2$ and 
 $y=\Psi(x)$. An easy computation shows that $\grd_x (\Psi^*f) =
 \grf(\grd_yf)$. Hence we have
 \begin{align*}
 \{\Psi^*f,\Psi^*g\}(x)
   &= \kappa_2\big( \grf(\grd_yf),Ad_x(\grf(\grd_yg)) \big) -
     \kappa_2\big( \grf(\grd_yg),Ad_x(\grf(\grd_yf)) \big) \\
   &= \kappa_1\big( \grd_yf,Ad_y(\grd_yg) \big) -
     \kappa_1\big( \grd_yg,Ad_y(\grd_yf) \big)
   = \{f,g\}(y)
 \end{align*}
 proving the claim.
 \end{proof}

 We apply  this result to our situation.  Recall some simple facts
 about the definition of the Poisson dual of a group $G$.  Fix a
 maximal torus $T$ of $G$.  Let $\gog$ be the Lie algebra of $G$ and
 $\got$ the Lie algebra of $T$. Fix an invariant non degenerate bilinear
 form $(-,-)$ on $\gog$.
 On the Lie algebra $\gog \oplus \gog$  
 define  the non degenerate bilinear form
 $\kappa\big((x,y),(u,v)\big)=(x,u)-(y,v)$.  The
 triple $(\gog \oplus \gog, \Delta, \goh)$ where
 $\Delta$ is the diagonal subalgebra and $\goh =\{(x+t,y-t)\in \gob
 \oplus \gob^- \st x, \in \gou, y \in \gou^-$ and $t \in \got\}$ is a Manin triple.  Indeed it is immediate to see 
  that both subalgebras are isotropic and that they are 
 disjoint. Correspondingly the Poisson dual of $G$ is the group
 $$ H=\{(x,y)\in B\times B^- \st \pi_T(x)=\pi^-_T(y)^{-1}\}$$ where
 $\pi_T,\pi_T^-$ are the projections on $T$ of $B,B^-$ respectively.
 As we have explained above the Manin triple induces on $H$ the
 structure of a Poisson Lie group and, if one consider the map
\begin{equation}\label{ladua}\rho=\rho_G:H\lra G\end{equation}
defined by $\rho(x,y)=x\,y^{-1}$,   the symplectic leaves in $H$ are
the connected components of the pre-images of the conjugacy classes in
$G$ (indeed in the simply connected case in \cite{DKP} it is shown that such
pre-images are always connected unless they are zero
dimensional). Furthermore the map $\rho$ is a fiber bundle onto the
open Bruhat cell of $G$ with fiber $\{t\in T \st t^2=1\}$. Let
$q:G\lra G/\!/Ad(G)$ be the quotient under the adjoint action and let
$\theta=\theta_G = q \circ \rho : H \lra G/\!/Ad(G)$. We denote the
algebra $\theta^*(\mk[G/\!/Ad(G)])$ by $Z^{HC}$ or $Z^{HC}_G$. It is
not difficult to check that each Hamiltonian vector field kills the
elements in $Z^{HC}$ so that this algebra is central with respect to
the Poisson structure.  

 \medskip

Consider now a standard Levi subgroup $L$ of $G$ and denote by $\gol$
its Lie algebra. Let $B_L=L\cap B$ and $B_L^-=L\cap B^-$ be the
standard Borel subgroup and the opposite standard Borel subgroup of
$L$, and denote by $\pi_L:B\lra B_L$, and $\pi^{-}_L:B^{-}\lra
B^{-}_L$ the projections onto the Levi factor.  Notice that the
restriction of the form $(-,-)$ to $\gol$ is non degenerate so we can
define a Manin triple $(\gol\oplus \gol , \Delta_\gol, \goh_\gol)$
taking the intersection of $\gol\oplus \gol$ with $\Delta$ and
$\goh$. Define also $H_L= H\cap L\times L$ and $\rho_L:H_L \lra L$ and
$Z^{HC}_L \subset \mk[H_L]$ as before. We notice that the transpose
$\psi$ of the inclusion, from $\goh$ to $\goh_L$ is a morphism of Lie
algebras and can be integrated to a map $\Psi_L:H\lra H_L$ given by
$\Psi_L(u,v)=(\pi_L^+(u),\pi_L^-(v))$. We can apply Lemma
\ref{lem:Manintriple} and we get that $\Psi_{L}^*$ is a morphism of
Poisson Lie groups. In particular $A^G_L = \Psi_L^*(Z^{HC}_L)$ is a Poisson
commutative subalgebra of $\mk[H]$. In the case that $L=T$, we can
define in a similar way a   Poisson commutative
subalgebra larger than $A^G_T$. In this case we can identify $H_T$ with $T$
and the Manin triple $(\got\oplus\got,\Delta_\got,\goh_\got)$ is
commutative so $H_T=T$ is Poisson commutative. Hence $A'_T=
\Psi_T^*(\mk[T])$ is a Poisson commutative subalgebra of
$\mk[H]$. Notice that $A'_T$   is an
extension of degree $2^{\dim T}$ of $A^G_T$.

\medskip

Let now $\calL=\{L_1\subset \dots L_h\}$ be a ss-admissible sequence
of standard Levi subgroups and let $\calM=\{M_1\subset \dots M_h\}$ be
an admissible sequence compatible with $\calL$. Choose subtori $S_i$
which satisfy properties S1 and S2.  Define $S=\prod_{i=1}^hS_i$. 
Choose a complement of $S$ in $T$ and denote by $p_S:T\lra S$ the
associated projection.  Let also, as in section \ref{ssec:ss}, denote by $p_i : L_i \lra
L_i^{ss}/\!/Ad(L_i^{ss})$   the projection on the adjoint quotient of the
semisimple factor of $L_i$. Define
$\theta_i = p_i \circ \rho_{L_i} \circ \Psi_{L_i} : H \lra
L^{ss}_i/\!/Ad(L^{ss}_i)$ and $\theta_\calM: H \lra S \times
\prod_{i=1}^h L_i/\!/Ad(L^{ss}_i)$ by
$$
\theta_\calM(u,v) =
\left(p_S\big(\Psi_T(u)\big),\theta_1(u,v),\dots,\theta_h(u,v) \right).
$$ We can use this map and the simultaneous Steinberg section to construct a
big Poisson commutative subalgebra of $\mk[H]$.

\begin{theorem} \label{teo:app1}
 The map $\theta_\calM ^* : \mk[S]\otimes \bigotimes_{i=1}^h \mk[L^{ss}_i]^{Ad(L^{ss}_i)} \lra
 \mk[H]$ is injective and its image is a Poisson commutative
 subalgebra of $\mk[H]$.
\end{theorem}
\begin{proof}
 In order to show the  injectivity of the map $\theta_\calM^*$ we are going to see that 
  the map $\theta_\calM$ is surjective. We can use the section $\chi$
constructed in Lemma \ref{lem:simultanea}. Using the notation of Section 2, let $A, C$ be as in Lemma \ref{lemma:Stss} and define $\chi':
 S \times \prod_{i=1}^h \mk^{\Delta_i}\lra H$ by
 $$
 \chi'(s,\undera) = \big( s\, A(s,\undera) ,\, s^{-1}C(s,\undera)^{-1}\big).
 $$ Then $\theta_\calM \big(\chi'(s,\undera)\big)= \Big(s,
 q_\calM^{ss}\big(\chi(s^2,\undera)\big) \Big)$. Since by Lemma
 \ref{lem:simultanea} for every fixed $s$ the map $a\mapsto
 q_\calM^{ss}\chi(s^2,\undera)$ is bijective we get that also the map
 $\theta_\calM\circ \chi'$ is bijective.

To prove that the image is a Poisson commutative subalgebra we notice that it
is contained in the product of the subalgebras $A^G_{L_1},\dots,
A^G_{L_h}$ and $A'_T$. Recall  $A^G_L= \Psi_L^* (Z^{HC}_L)$   is Poisson commutative. So it is
enough to prove that that $A^G_{L_i}$ commutes with $A^G_{L_j}$ when $i>j$ and with  $A'_T$. Notice
that if $i>j$, then $A^G_{L_j}= \Psi_{L_i}^*(A^{L_i}_{L_j})$. Hence,
since by Lemma \ref{lem:Manintriple} $\Psi_{L_i}^*$ is a Poisson map
and $Z^{HC}_{L_i}$ is in the center of $\mk[H_{L_i}]$, $A^G_{L_j}$
commutes with $A^G_{L_i}$. We can argue in a similar way for the algebra $A'_T$.
\end{proof}



In the case of $GL(n)$ we can produce in this way a commutative
subalgebra of maximal dimension. We consider again the Gelfand-Zeitlin
admissible sequence
$$\calM =\{ GL(1)\subset \dots \subset GL(n)\}. $$ In this case the product $S
\times \prod_{i=1}^n SL(i)/\!/Ad(SL(i))$ has dimension
$\binom{n+1}{2}$ and its coordinate ring $A$ is the localization of a
polynomial algebra in $\binom{n+1}{2}$ with respect to some variables.

By what we have recalled above, the generic symplectic leaves of $H$ have
dimension equal to the regular orbits in $G$, hence to $n^2-n$. So
maximal isotropic subspaces in the tangent space of a generic point of
$H$ have dimension $n^2-\binom{n}{2}=\binom{n+1}{2}$.

Generically the quotient map $\theta_\calM$ is a smooth map and its fibers are
maximal dimensional isotropic sub-varieties of $H$. So the dimension of
$A$ is the maximal possible dimension of a commutative Poisson
subalgebra of $\mk[H]$, which can be stated by saying that it defines
a completely integrable Hamiltonian system.

\section{The center of a quantum group in the reductive case}\label{sec:centro}
In this section we recall the description of the center of a quantum
group at roots of unity in the simply connected case proved in
\cite{dp}. We also give an extension of this result to the case of a
reductive group which we could not find in the literature. We assume
the characteristic of $\mk$ to be equal to zero from now on.

We start by giving the definition of the quantum group associated to a
reductive group. If $G$ is a connected reductive group we denote by
$T_G$ a maximal torus of $G$ and by $\Lambda_G$ the   lattice of
characters of the chosen maximal torus $T_G$.  We choose a set of
simple roots $\Delta_G=\{\alpha_1,\ldots ,\alpha_r\}$ and we set
$a_{i,j}=\langle \alpha_i,\check \alpha_j \rangle$ so that
$C=(a_{i,j})$ is the Cartan matrix. Let $(d_1,\ldots d_r)$ be the
usual non zero entries of the diagonal matrix such that $CD$ is
symmetric. We also assume that we have a non degenerate symmetric
invariant bilinear form $(-\, ,\, -)$ on the Lie algebra of $G$.  If
we restrict $(-\, ,\, -)$ to the Cartan subalgebra $\mathfrak t$= Lie
$T_G$, we get a non degenerate form. We assume that the corresponding
form on $\mathfrak t^*$ takes integer values on $\Lambda_G\times
\mZ[\Delta_G]$ and furthermore $CD=\big((\alpha_i,\alpha_j)\big)$.
Notice that $\langle \grl,\check\alpha\rangle=0$ if and only if
$(\grl,\alpha)=0$. We set $q_i=q^{d_i}$.

For a non zero complex number $q$ the algebra $U_q(G)$ is the algebra
with generators $E_1,\ldots ,E_n;F_1,\ldots ,F_n$ and $K_{\lambda}$,
$\lambda\in \Lambda_G$ and relations:

\begin{eqnarray*}
 &(R1)&\quad K_{\lambda} K_{\mu} =K_{\lambda+\mu} , \quad
  K_{0} =1,\\
 &(R2)&\quad K_{\lambda}E_iK_{-\lambda}=q^{( \lambda,\alpha_i )} E_i,\\
 &(R3)&\quad K_{\lambda}F_iK_{-\lambda}=q^{-( \lambda,\alpha_i )} F_i,\\
 &(R4)&\quad[\,E_i, F_j\,]=\delta_{i,j}\frac{K_{\alpha_i}-K_{-\alpha_i}}{q_i-q_i^{-1}},\\
 &(R5)&\quad \sum_{k=0}^{1-a_{ij}} (-1)^{k}
 \left[ \begin{array}{c} 1-a_{ij}\\ k\end{array}\right]_{q_{i}}E_{i}^{1-a_{ij}-k} E_{j} E_{i}^{k} =0 \quad (i\ne j),\\
 &(R7)&\quad\sum_{k=0}^{1-a_{ij}} (-1)^{k}
 \left[ \begin{array}{c} 1-a_{ij}\\ k\end{array}\right]_{q_{i}}F_{i}^{k} F_{j} F_{i}^{1-a_{ij}-k}=0 \quad (i\ne j),
\end{eqnarray*}
for all $\lambda,\mu\in \Lambda$, $i,j=1,\ldots r$ where we set for
any $t$, and for $h\leq m$,
$$\left[ \begin{array}{c} m\\h \end{array}\right]_t\, =
\frac{[m]_t!}{[m-h]_t![h]_t!} \ ,\qquad [h]_t!=[h]_t\dots [2]_t[1]_t
\ , \qquad \left[h\right]_t = \frac{t^h-t^{-h}}{t-t^{-1}} \ . $$
Recall that $U_q(G)$ is a Hopf algebra (see \cite{dr}).
We denote the center of $U_q(G)$ by $Z_q(G)$.

\medskip

If we have an isogeny  $\phi: G_1 \lra G_2$ then this
determines an inclusion of $\Lambda_{G_2}$ in $\Lambda_{G_1}$ and of
$U_q(G_2)$ in $U_q(G_1)$. Let $\Gamma$ be the kernel of $\phi$ and
notice that $\Gamma = \Spec \mk[\Lambda_{G_1}/\Lambda_{G_2}] =
\Hom(\Lambda_{G_1}/\Lambda_{G_2},\mk^*)$ so we can identify it with a group of
 characters of $\Lambda_{G_1}$. We have the following action
of $\Gamma$ on $U_q(G_1)$:
$$
\gamma \, E_i= E_i \qquad \;
\gamma \, F_i= F_i \qquad \;
\gamma \, K_\lambda = \gamma (\lambda) K_\lambda
$$ for all $\lambda \in \Lambda_{G_1}$ and for all $i$. Notice that
since $\phi$ is an isogeny all the roots are contained in
$\Lambda_{G_2}$, hence the action is well defined.  Notice also that
$K_\mu $ is fixed by $\Gamma$ if and only if $\mu$ is an element of
$\Lambda_{G_2}$. Hence using the Poincar\'e Birkoff
Witt basis \cite{lus}, we obtain the following Lemma.

\begin{lemma} \label{lem:UGamma}
With the notation above, we have
$$ U_q(G_2) = U_q(G_1)^\Gamma $$ and $U_q(G_1)$ is a free module over
$U_q(G_2)$ with basis given by $K_\mu$ with $\mu$ which vary in a
set of representatives of $\Lambda_{G_1} /\Lambda_{G_2}$.
\end{lemma}

 Let $\tilde S_G$   be the
connected component of the center of our reductive group $G$. Then $\tilde S_G=M\otimes_{\mathbb Z}
\mk^*$ where $M$ is the lattice of cocharacters which are trivial on
$\Delta_G$. In the previous setting choose $G_1=G^{ss}\times \tilde
S_G$, $G_2=G$ and $\phi$ the multiplication map.  The kernel $\Gamma$ is
the set of pairs $(\gamma,\gamma^{-1})$ with $\gamma\in \tilde S_G \cap
G^{ss}$ and we will identify it with a subgroup of $G$.

In this case we have $\Lambda_{G_1} = \Lambda_{G^{ss}} \oplus
\Lambda'$, where $\Lambda'$ is the dual of $M$. We have a surjective
projection from $\Lambda_G$ to $\Lambda_{G^{ss}}$ given by
restriction of characters, whose kernel is equal to $N=\Delta^\perp$
and a surjective projection from $\Lambda_G$ to $\Lambda'$ whose
kernel is equal to $N^\perp$.

Since the algebra $U_q(G^{ss}\times \tilde S_G)$ is clearly isomorphic
to $U_q(G^{ss})\otimes \mk [\tilde S_G]$, applying Lemma
\ref{lem:UGamma}, we deduce that $U_q(G)$ can be identified with
$\Big (U_q(G^{ss})\otimes \mk [\tilde S_G]\Big )^{\Gamma}.$

We now start to describe the center $ Z_q (G) $ of $U_q(G)$. We have

\begin{proposition} \label{prp:ZGamma}
Under the identification of $U_q(G)$ with $(U_q(G^{ss})\otimes \mk
[\tilde S_G])^{\Gamma},$  
$$ Z_q (G) =(Z_q(G^{ss}) \otimes \mk[\tilde S_G])^{\Gamma}= Z_q
(G^{ss}\times \tilde S_G)^\Gamma .$$ Moreover $Z_q (G^{ss}\times
\tilde S_G)$ is a free module over $Z_q(G)$ with  a basis   given by the elements
$K_\mu$ with $\mu$   varying in a set of representatives of
$(\Lambda_{G^{ss}}\oplus \Lambda') /\Lambda_G$.
\end{proposition}

\begin{proof}
It is   clear that $ Z_q (G) \supset Z_q (G^{ss}\times \tilde
S_G)^\Gamma $.  We prove the other inclusion. Notice that we can choose
a set $A$ of representatives of $(\Lambda_{G^{ss}}\oplus \Lambda') /\Lambda_G$ in the lattice
$\Lambda'$.    Take $x \in U_q(G^{ss}\times
\tilde S_G)$. By Lemma \ref{lem:UGamma}, $x$  can be written as $x=\sum_{\mu} x_\mu K_\mu$ with $\mu \in A$ and $x_\mu\in U_q(G)$. Since the
$K_\mu$ are central,   if $z \in Z_q(G)$, then it commutes
also with $x$. Hence $z \in Z_q (G^{ss}\times \tilde S_G)$.

Our argument also proves that the elements $K_\mu$ with $\mu \in A$ are a basis of $
Z_q (G^{ss}\times \tilde S_G)$ over $Z_q(G)$.
\end{proof}

We now give a geometric description of $Z_q(G)$ in the case $G$ has
simply connected semisimple factor and $q$ is a primitive $\ell$-th
root of $1$ with $\ell$ prime with the entries of the Cartan matrix. This description
generalizes the one given in \cite{dprr}. Set $X_G=\Spec Z_q(G)$. We
will use the notations introduced in sections \ref{sez:app1} and
\ref{sez:stein}. So $H_G=H$ is the Poisson dual of $G$,
$\rho_G=\rho:H_G\lra G$ is the map $\rho_G(u,v)=uv^{-1}$, $W_G$ is the
Weyl group of $G$ w.r.t.\ $T_G$, $q_G^{ss}:G\lra G^{ss}/\!/Ad(G^{ss})
= T_{G^{ss}}/W_G$ is an extension of the adjoint quotient of $G^{ss}$
as constructed in section \ref{ssec:Steinbergred}. If we set $p_{S_G}:
G\lra S_G $ to be  the quotient of $G$ by its semisimple factor, then
$q_G=q^{ss}_G\times p_{S_G} : G \lra (T_{G^{ss}}/W_G) \times S_G$ is the
quotient map under the adjoint action and  
$\theta_G=q_G\circ\rho_G$. Notice that $p_{S_G}$ induces an isogeny from
$\tilde S_G$ and $S_G$ whose kernel is equal to $\Gamma$ and that, as
we did in section \ref{ssec:Steinbergred}, we can identify $S_G$ with a
complement of $T_{G^{ss}}$ in $T$. When necessary (mainly for the
construction of the Steinberg section as in Section
\ref{ssec:Steinbergred}) we will identify $S_G$ with this
torus. Finally, we denote by $\eta_G:T_G/W_G\lra T_G/W_G$ the map
induced by $t\mapsto t^\ell$ from $T_G\lra T_G$. This is a finite flat
covering of $T_G/W_G$ and it is an unramified Galois covering over the set
$(T_G/W_G)^r=T_G^r/W_G$ where $T_G^r$ is the set of regular elements
of $T_G$.

Recall, \cite{DK}, that $Z_q(G)$ is endowed with a Poisson bracket and contains
the elements $E_i^\ell, F_i^\ell, K_\lambda^\ell=K_{\ell\lambda}$, for
all $i=1,\ldots r$, $\lambda\in\Lambda.$ Let $Z_0(G)$ be the smallest
subalgebra of $U_q(G)$ closed under the Poisson bracket and containing
the elements $E_i^\ell, F_i^\ell, K_\lambda^\ell$. 

\medskip

We now recall the description of $X_G$ given in \cite{dprr} in the
simply connected case. In this case the subalgebra $Z_0(G)$ is
isomorphic to the coordinate ring of the Poisson dual $H_G$ of $G$.
There exists also another subalgebra $Z_1(G)$ in $Z_q(G)$ which is
isomorphic to the ring of functions on $T_G$ invariant under the
action of $W_G$. The Poisson bracket is trivial on
$Z_1(G)$.

The map $\eta_G$ induces a map $\eta_G^*$ from $Z_1(G)$ to $Z_1(G)$.
We denote by $Z_1^{(\ell)}(G)$ its image. Identifying
$\mk[T_G/W_G]$ with $Z_1^{(\ell)}(G)$ and $\mk[H_G]$ with $Z_0(G)$ we
get an inclusion $Z_1^{(\ell)}(G)\subset Z_0(G)$. One
then shows (see e.g. \cite{dp}) that $$Z_q(G)\simeq
Z_0(G)\otimes_{Z_1^{(\ell)}(G)}Z_1(G).$$ Equivalently, setting
$X_G=\Spec Z_q(G)$, if $G$ is simply connected, the following diagram
is cartesian
\begin{equation}\label{eq:Zsc}
\begin{CD} 
    X_G @>{\zeta^1_G}>> T_G/W_G \\
    @V{\zeta_G}VV @VV{\eta_G}V \\
    H_G @>{\theta_G}>> T_G/W_G
\end{CD}
\end{equation}
where we denoted by $\zeta_G:X_G\to H_G$ the finite morphism of degree
$\ell^{\dim T_G}$ induced by the inclusion of $Z_0(G)$ in $Z_q(G)$ and
by $\zeta_G^1:X_G\lra T_G/W_G$ the one induced by the inclusion of $Z_1(G)$ in
$Z_q(G)$.

\medskip

We now give a similar description in the case in which the semisimple
factor of $G$ is simply connected. Before stating our result, we make
some remarks on the action of $\Gamma$ on the objects introduced so
far.

We consider the actions of $\Gamma$
on $H_{G^{ss}}\times \tilde S_G$ and on $(T_{G^{ss}}/W_G) \times \tilde
S_G$ given by 
$$
\gamma\cdot \big((u,v),s\big) = \big(\gamma\, u,\gamma^{-1} v, \gamma^{-1}s) 
\quad \mand \quad 
\gamma\circ (tW_G,s) = (\gamma^2tW_G, \gamma^{-1}s).
$$ With these actions, the map $\theta_{G^{ss}} \times id:
H_{G^{ss}}\times \tilde S_G \lra (T_{G^{ss}}/W_G) \times \tilde S_G$ is
equivariant while the map $\eta_{G_1}$ satisfies $\eta_{G_1}(\gamma
\circ x)= \gamma^\ell \circ \eta_{G_1}(x)$. Moreover, we consider
$\phi_0:H_{G^{ss}}\times \tilde S_G\lra H_G$ given by
$\phi_0\big((u,v),s\big) = \big(us,vs^{-1}\big)$ and we notice that
this is the quotient map by the action of $\Gamma$.

By what we have said about the semisimple case, we can identify the
coordinate ring of $H_{G^{ss}}\times \tilde S_G$ with $Z_0(G^{ss}\times
\tilde S_G)$, where $K_{\ell \lambda}$, for $\lambda \in \Lambda'$,
corresponds to the character $\lambda$ on $\tilde S_G$. Then the
 action of $\Gamma$ corresponds to the following action  on
$Z_0(G^{ss}\times \tilde S_G)$:
$$
\gamma \cdot E^{\ell}_i= E^{\ell}_i \qquad \;
\gamma \cdot F^{\ell}_i= F^{\ell}_i \qquad \;
\gamma \cdot K_{\ell \lambda} = \gamma (\lambda) K_{\ell\lambda}
$$ for all $\lambda \in \Lambda_{G_1}$ and for all $i$.  Notice that
with this action we have $ Z_0(G)= (Z_0(G^{ss}\times \tilde S_G))^{\Gamma}
$. Hence 
$$ \Spec Z_0(G) \simeq (H_{G^{ss}} \times \tilde S_G) /\Gamma \simeq
H_G $$ where the last isomorphism is given by $\phi_0$.

Now we want to describe the subalgebra $Z_1(G) := (Z_1(G^{ss})\otimes
\mk[\tilde S_G])^\Gamma$ where we are now considering the usual action
of $\Gamma$ on $U_q(G^{ss}\times \tilde S_G)$. Notice that this makes
sense since the subring $Z_1(G^{ss}) \otimes \mk[\tilde S_G]$ is stable
under $\Gamma$.

\begin{lemma}\label{lem:Z1}
There exists $\phi_1:(T_{G^{ss}}/W_G)\times \tilde S_G \lra (T_{G^{ss}}/W_G) \times S_G$
such that the following properties hold:
\begin{enumerate}[\indent i)]
\item $\phi_1$ is $\Gamma$ invariant with respect to the $\circ$-action,
  more precisely it is the quotient of $(T_{G^{ss}}/W_G)\times \tilde S_G$ by the $\circ$-action;
\item $\phi_1(x,s) = (\psi(x,s),p_{S_G}(s))$ with $\psi(x,s) \in T_{G^{ss}}/W_G$;
\item the following diagram is commutative and cartesian
\begin{subequations}\label{eq:diagrammi}
\begin{equation}\label{eq:d1}
\begin{CD}
H_{G^{ss}}\times \tilde S_G @>{\theta_{G^{ss}} \times id }>> (T_{G^{ss}} /W_G) \times \tilde S_G  \\
@V{\phi_0}VV @VV{\phi_1}V \\
H_G @>{\theta_{G}}>> (T_{G^{ss}}/W_G)\times S_G;
\end{CD}
\end{equation}
\item the following diagram is commutative
\begin{equation}\label{eq:d2}
\begin{CD}
(T_{G^{ss}}/W_G)\times \tilde S_G 
@>{\eta_{G^{ss}}\times \eta_{\tilde
    S_G}}>> T_{G^{ss}} /W_G \times \tilde S_G \\ @V{\phi_1}VV
@VV{\phi_1}V \\ (T_{G^{ss}}/W_G)\times S_G @>{\eta_G}>>
(T_{G^{ss}}/W_G)\times S_G.
\end{CD}
\end{equation}
\end{subequations}
\end{enumerate}
\end{lemma}

\begin{proof}
Let $\xi_1,\dots,\xi_n$ be the characters of the fundamental
representations of $G^{ss}$.  Then the coordinate ring of
$T_{G^{ss}}/W_G$ is the polynomial ring in $\xi_1,\dots,\xi_n$.  As in
section \ref{ssec:Steinbergred} we can extend these characters to get
characters of $G$ that we denote with $\xi'_1,\dots, \xi_n'$. Notice
that by definition the map $q_G^{ss}$ is given by evaluating
$(\xi'_1,\dots,\xi'_n)$. Now define $f_i: G^{ss} \times \tilde S_G\to \mk$
 by
$$
f_i(x,s) = \xi_i(x)\xi'_i(s^2). 
$$
Then $f_i$ is $\Gamma$ invariant and  
$$\mk[G^{ss}/\!/Ad(G^{ss}) \times \tilde S_G]\simeq
\mk[f_1,\dots,f_n]\otimes \mk[\tilde S_G]$$ 
since $\xi'_i(s)$ is a character of $\tilde S_G$. 
Taking invariants, we obtain 
$$
\mk[G^{ss}/\!/Ad(G^{ss}) \times \tilde S_G]^{(\Gamma,\circ )}\simeq
\mk[f_1,\dots,f_n]\otimes \mk[S_G].
$$ 
Now define $\phi_1$ by
$$
\phi_1(x,s ) = \big(\big(f_1(x,s),\dots,f_n(x,s), p_{S_G}(s) \big).
$$ By the description of the invariants $\phi_1$  satisfies $i)$. 

$ii)$
is clear by definition. 

To prove $iii)$, notice  that for all
$\big((u,v),s\big) \in H_{G^{ss}}\times \tilde S_G$, we have
$$
\theta_G\left(\phi_0\big((u,v),s\big)\right)=
\theta_G(us,vs^{-1})=\big(q_{G}^{ss}(us^2v^{-1}),p_{S_G}(s)\big)
$$
and, since $\tilde S_G$ is central,
\begin{align*}
q_{G}^{ss}(us^2v^{-1}) &= \big(\xi'_1(uv^{-1}s^2),
\dots,\xi'_n(uv^{-1}s^2)\big) \\ &=
\big(f_1(q_{G^{ss}}(uv^{-1}),s),\dots,f_n(q_{G^{ss}}(uv^{-1}),s)\big)
\end{align*}
from which   the commutativity of diagram \eqref{eq:d1} follows.
Since the vertical maps in that diagram are   quotients by $\Gamma$, the varieties are
smooth and the action on the the varieties in the top line is free we
get that it is also cartesian.

The commutativity of \eqref{eq:d2} is clear.
\end{proof}

The previous Lemma implies that $\Spec Z_1(G) \simeq (T_{G^{ss}}/W_G) \times S_G$.  As we have done in 
the simply connected case, we denote by $\zeta_G:X_G\lra H_G$, and 
$\zeta^{1}_G:X_G \lra (T_{G^{ss}}/W_G) \times S_G$ the maps 
induced by the inclusion $Z_0(G)\subset Z_q(G)$ and 
$Z_1(G)\subset Z_q(G)$ respectively.

We can now give our second description of the center of $U_q(G)$.

\begin{proposition}\label{prp:Z}
Let $G$ be reductive with simply connected semisimple factor. Then
the following diagram is cartesian
\begin{equation}\label{eq:Z}
 \begin{CD} X_G @>{\zeta^1_G}>> (T_{G^{ss}}/W_G) \times S_G\\
  @V{\zeta_G}VV @VV{\eta_G}V \\
 H_G @>{\theta_G}>> (T_{G^{ss}}/W_G) \times S_G
 \end{CD}
\end{equation}
\end{proposition}

\begin{proof}
If $G=G^{ss}\times \tilde S$ with $\tilde S$ a torus then $\tilde S$
is also the connected component of the center of $G$ and by
Proposition \ref{prp:ZGamma}, we immediately have that
$Z_q(G)=Z_q(G)\otimes \mk [\tilde S]$. Hence $X_G= X_{G^{ss}} \times \tilde S$.  We
can write a diagram similar to \eqref{eq:Zsc}.  Notice that in this
case $H_G=H_{G^{ss}}\times \tilde S$, $T_G=T_{G^{ss}}\times \tilde S$ and
$T_G/W_G=(T_{G^{ss}}/W_G) \times \tilde S$. Also the $\ell$ power map $\eta_G$
is the product of the two power maps $\eta_{G^{ss}}$ and $\eta
_{\tilde S}:\tilde S\lra \tilde S$. Then we have the following cartesian diagram:
\begin{equation*}\label{eq:Zprodotto}
\begin{CD}
X_G@>>>  (T_{G^{ss}}/W_G) \times \tilde S\\
@V{\zeta_G}VV    @VV{\eta_{G^{ss}} \times \eta_S}V\\
H_{G^{ss}}\times \tilde S @>{\theta_{G^{ss}}\times id}>> (T_{G^{ss}}/W_G) \times \tilde S.
\end{CD}
\end{equation*}
So, in this case, everything is an immediate consequence of what we know for $G^{ss}$.

\medskip

Now we consider the case of an arbitrary reductive group $G$ with
simply connected semisimple factor. Let  $\phi: G^{ss}\times \tilde S_G \lra G$ be the
multiplication map. By Proposition \ref{prp:ZGamma} we have that $X_G=
X_{G^{ss}\times \tilde S_G}/\Gamma$. We denote by $Y$ the pull back of the maps $\eta_G,
\theta_G$ and we prove that it is isomorphic to $X_G$.

Arguing exactly as in \cite{dprr}, one sees that $Y$ is irreducible,
normal, Cohen-Macaulay, and the map $Y\lra H_G$ has degree $\ell^
{\dim T_G}$. By the commutativity of the diagrams in
\eqref{eq:diagrammi} we have a natural map $\psi: X_G \lra Y$.


Moreover the morphism $\psi$ is finite and since the morphism
$\zeta_G:X_G\lra H_G$ has also degree $\ell^ {\dim T_G}$ it is also
birational, hence it is an isomorphism.
\end{proof}

\section{Branching rules for quantum groups at roots of 1}
In this section we are going to show how to obtain some branching
rules for quantum groups at roots of $1$ following the ideas of
\cite{dprr}. In order to do this, let us recall the notion of a
Cayley-Hamilton algebra and some results on the representation
theory of quantum groups.

\subsection{Cayley-Hamilton algebras}
An algebra with trace, over a commutative ring $A$ is an associative
algebra $R$ with a 1-ary operation
 \[ t:R\to R \]
which is assumed to satisfy the following axioms:
 \begin{enumerate}
    \item  $t$ is $A-$linear.
    \item  $t(a)b=b\,t(a), \quad\forall a,b\in R.$
    \item  $t(ab)=t(ba), \quad\forall a,b\in R.$
    \item  $t(t(a)b)=t( a)t(b), \quad\forall a,b\in R.$
 \end{enumerate}
 This operation is called a formal trace.

The ring of $n\times n$ matrices with entries in $A$ with the usual trace is an example. 

 At this point we will restrict the discussion to the case in which $A$
 is a field of characteristic 0. We remark that there are universal
 polynomials $P_i(t_1,\dots,t_i)$ with rational coefficients, such
 that the characteristic polynomial  $\chi_M(t):=\det(t-M)$ of a $n\times n$ matrix $M$ can be written as:
 $$ \chi_M(t)=t^n+\sum_{i=1}^nP_i(tr(M),\dots,tr(M^i))t^{n-i}. $$

 We can thus formally define, in an algebra with trace $R$, for every
 element $a$, a formal $n-$characteristic polynomial:
 \[
 \chi_a^n(t):=t^n+\sum_{i=1}^nP_i(t (a),\dots,t (a^i))t^{n-i}.
 \]
 The Cayley-Hamilton Theorem, stating  that a matrix  $M$  satisfies its
 characteristic polynomial, suggests  the
  following definition.

\begin{definition}(\cite{dprr} Definition 2.5)
 An algebra with trace $R$ is said to be an $n$-Cayley-Hamilton algebra,
 or to satisfy the $n^{th}$ Cayley-Hamilton identity if:

 1) $t(1)=n$.

 2) $\chi_a^n(a)=0, \ \forall a\in R$.
\end{definition}

In \cite{dprr} , Theorem 4.1. it is proved that if a $\mk$ algebra
$R$ is a domain which is a finite module over its center $A$, and
$A$ is integrally closed in his quotient field $F$, then $R$ is a
$n$-Cayley-Hamilton algebra with $n^2=\dim_F R\otimes_AF$.

As usual, let now $G$ be a reductive group with simply connected semisimple
factor, $\ell$ be a natural number prime with the entries of the
Cartan matrix and $q$ a primitive $\ell$-th root of unity.

The algebra $U_q(G)$ is a domain. Indeed this is shown in \cite{DK},
\cite{DKP} for $G$ semisimple and follows for a general reductive $G$
from the fact that $U_q(G)=(U_q(G^{ss})\otimes \mk [\tilde
  S_G])^{\Gamma}.$ Furthermore the fact that the center $Z_q(G)$ is
integrally closed is immediate from Proposition \ref{prp:ZGamma} since
$ Z_q (G) =(Z_q(G^{ss}) \otimes \mk[\tilde S_G])^{\Gamma}$ and
$Z_q(G^{ss}) \otimes \mk[\tilde S_G]$ is integrally closed. Thus
$U_q(G)$ is a $n$-Cayley-Hamilton algebra and by Lemma
\ref{lem:UGamma} and Proposition \ref{prp:ZGamma} it follows that  $n=\ell^{|\Phi^+|}$.

As a consequence (see \cite{dprr} Theorem 3.1), the variety
$X_G$ parametrizes semisimple representations compatible
with the trace (i.e. representation of $U_q(G)$ such that the formal
trace coincide with the trace computed by considering the action of
the element of $U_q(G)$ on the representation). Furthermore there is a
dense open set $X_G^{ir}$ of $X_G$ such that for any $x\in X_G^{ir}$ the
corresponding trace compatible representation is   irreducible
and it is the unique irreducible representation on which $Z_q(G)$ acts
via the evaluation on $x$. We denote this representation by $V_x(G)$.

We define also $G^{sr}=q_G^{-1}((T/W)^r)$ to be the open set of
semisimple elements of $G$, $H_G^{sr}=\rho_G^{-1}(G^{sr})$ and
$X_G^{sr}=\zeta^{-1}(H^{sr})$. Hence by Proposition \ref{prp:Z}
$\zeta_G:X^{sr}_G\lra H_G^{sr}$ is an unramified covering of degree
$\ell^{\dim T_G}$. Finally by \cite{DKP}, it is known that
$X_G^{sr}\subset X_G^{ir}$.

\subsection{Branching rules}
Let $L$ be a standard Levi subgroup of $G$ and let $M\subset G$ be
an admissible subgroup compatible with $L$ with the property that
there exists a homomorphism $\grs:L\to M$ which splits the inclusion
$M\subset L$.  

This is satisfied in the following two examples which are the main
applications we have in mind. The first is when $G$ is semisimple
simply connected and $L=M$. The second is when $G=GL(n)$,
$$L=\left\{\left(\begin{array}{cc}A & 0 \\0 & b\end{array}\right)\, |
A\in GL(n-1),\ b\in \mathbb C^*\right\},$$ and
$$M=\left\{\left(\begin{array}{cc}A & 0 \\0 & 1\end{array}\right)\, | A\in GL(n-1)\right\}.$$

Let $T_M = T \cap M$ be a maximal torus of $M$ and let $R= \ker
\sigma$ be such that $T=R\times T_M$. We can identify the algebra
$U_q(M)$ with the subalgebra of $U_q(G)$ generated by the elements
$K_{\lambda}$, $\lambda$ vanishing on $R$, and by the element $E_i,
F_i$ with $\alpha_i\in \Delta_L$.  Our goal is to describe how
``generic" irreducible representations of $U_q(G)$ decompose when
restricted to $U_q(M)$.

The discussion of the previous section about the representations and
the center of $U_q(G)$ applies  to $U_q(M)$ as well.  Hence $U_q(M)$ is
itself a $n_L$-th Cayley-Hamilton algebra, this time with
$n_L=\ell^{|\phi_L^+|}$.  

We define $\Psi_M:H_G \lra H_L $ as 
$$
\Psi_M(u,v) = \left( \grs \left( \pi_L^+(u)\right) , \grs \left( \pi_L^-(v)\right) \right).
$$ 
Then, under the identification of $Z_0(G)$ and $Z_0(M)$ with the
coordinate rings of $H_G$ and $H_M$ respectively,     the inclusion of
$Z_{0}(M)$ in $Z_0(G)$ corresponds to the map $\Psi_M^*$.

For $x \in X_G$ we define

$$\calM_x = \{y\in X_M \st \zeta_M(y)=\Psi_M(\zeta_G(x))\}.$$ 

\begin {theorem} \label{branch} Let $x\in X_G^{ir}$ be such that 
$\Psi_M(\zeta_G(x))\in H_M^{sr}$.
Then, taking the associated graded for any Jordan-H\"older filtration
of $V_x(G)$ we get:
$$Gr \big(V_x(G)\big)|_{U_q(M)}\simeq \bigoplus_{y\in \calM_x}V_y(M)^{\bigoplus b}$$
with $b=\ell^{|\Phi^+|-|\Phi^+_M|-\dim T_M}$.
\end{theorem}
\proof In view of Lemma 5.7 in \cite {dprr}, our Theorem will follow
once we show

\begin{proposition}
The algebra $Z_q(G)\otimes_{Z_0(M)}Z_q(M)$ is an integral domain.
\end{proposition}
\begin{proof}
Set $\calB = H_G^{sr} \cap \Psi_M^{-1}(H_M^{sr})$ and $\calA =
\zeta_G^{-1}(\calB)$.

Following almost verbatim the proof of Proposition 7.4 in \cite{dprr}, 
one is reduced to show that   taking the fiber product
 $$\begin{CD}Y@>>> T_{M^{ss}}/W_M  \times S_M \\ @VVV @VV{\eta_M}V\\ \mathcal
   A@>{\theta_M\circ\Psi_M\circ\zeta_G}>> T_{M^{ss}}/W_M  \times S_M \end{CD}$$ the variety
   $Y$ is irreducible.  

To show the irreducibility of $Y$,   notice that we have  
the following cartesian diagram
 $$\begin{CD}
 Y @>>> (T_{M^{ss}}/W_M  \times S_M ) \times ((T_{G^{ss}}/W_G) \times S_G) \\
 @VVV @VV{\eta_M\times \eta_G}V \\
 \mathcal B @>{\tilde \theta}>> (T_{M^{ss}}/W_M  \times S_M ) \times ((T_{G^{ss}}/W_G) \times S_G) 
 \end{CD}$$
where $\tilde \theta$ is the restriction to $\calB$ of the map
$\theta=(\theta_M\circ \Psi_M)\times \theta_G$.  By the definition of
$H_G^{sr}$ and $H_M^{sr}$, the set $\tilde \theta(\calB)$ is contained in
$\calC=(T_M/W_M)^r\times (T_G/W_G)^r$, and the map $\eta_M\times
\eta_G$ over $\calC$ is an unramified covering of degree $\ell^{\dim
  T_G+\dim T_M}$. So the map $Y\lra B$ is an unramified covering of
smooth varieties. Hence, to prove that $Y$ is irreducible it is enough
to prove that it is connected. We prove this by giving a section of
$\theta: H_G \lra (T_{M^{ss}}/W_M  \times S_M ) \times ((T_{G^{ss}}/W_G) \times S_G)$.

Consider the admissible sequence $\calM =\{M\subset G\}$.  We can
choose $S_M$ and $S_G$ subtori of $T_M$ and $T_G$ satisfying
conditions S1 and S2 of Section \ref{ssec:simultanea}, and we identify
these tori with the quotients of $G$ by $G^{ss}$ and of $M$ by $M^{ss}$ respectively, 
using the maps $p_{S_G}$ and $p_{S_M}$. Set $S=S_M\times S_G$ and let
$\chi: S \times \mk^{\Delta_L} \times \mk^{\Delta_G} \lra B\times U^-$ be 
the generalized Steinberg section constructed in Theorem
\ref{teo:simultanea}. Recall that
$\chi(s,a,b)=\big(sA(s,a,b),C(s,a,b)\big)$ and that, as it is shown in
the proof of Theorem \ref{teo:simultanea} and it is clear by
construction, $r_\calM(\chi(s,a,b))= \big(\grf(s), \psi(s,a,b)\big) \in
S \times (T_{M^{ss}}/W_M \times T_{G^{ss}}/W_G) $. So, $\grf$ is an
isomorphism and for all $s$ the map $(a,b)\mapsto \psi(s,a,b)$ is a
bijection     between $\mk^{\Delta_L} \times \mk^{\Delta_G}$ and
$T_{M^{ss}}/W_M \times T_{G^{ss}}/W_G$.  

Now define $\chi':S \times
\mk^{\Delta_L} \times \mk^{\Delta_G} \lra H_G$ by 
$$
\chi'(s,a,b) = \big(sA(s^2,a,b) , s^{-1}C(s^2,a,b)^{-1} \big)
$$
and notice that 
$$ \theta \big( \chi'(s,a,b)\big) =
\big(\grf(s),\psi(s^2,a,b)\big). $$ So $\theta \circ \chi'$ is
bijective and, being the involved varieties smooth, it is an isomorphism.

This finishes the proof of the Proposition and hence of Theorem
\ref{branch}.
\end{proof}
\begin{remark}\hfill

 \begin{enumerate}

\item Notice that since the algebras $U_q(G)$ and  $U_q(G^{ss})$  have   the same degree, equal to $\ell^{|\Phi^+|}$, and we are dealing with generic irreducible representations, one could have assumed right away that $G$ is semisimple. 
\item If $G$ is semisimple with simply connected cover $\tilde G$ again the algebras $U_q(G)\subset U_q(\tilde G)$ have the same degree
and our result holds verbatim also for $U_q(G)$.
\item Assume $G$   semisimple.  When  $L=T_G$,  the algebra $U_q(L)$ is the algebra of functions on $T_G$ which has the  $K_\lambda$'s  as a basis. Each element in $T_G$ is a character for this algebra. Given a finite dimensional  $U_q(L)$ module   we can consider its character as a non negative, integer valued  function on $\Lambda_G$  of finite support. This applies in particular to   $U_q(G)$ modules. Consider the   map $\Psi_T:H_G\to T_G=H_{T_G}$   and take an irreducible $V$ lying over an element $h\in H_G^{sr}$. The set $\mathcal A:=\eta^{-1}(\Psi(h))$ has $\ell^{\dim T}$ elements.  Our result gives that the character of $V$  equals  $\ell^{|\Phi^+|-\dim T}$ times the characteristic function of $\mathcal A$.\end{enumerate}
\end{remark}

\subsection {The case of $GL(n)$}
We want to apply Theorem \ref{branch} in the special case in which
$G=GL(n)$ and $M=GL(n-1)$.  We keep the notations of the previous
section
\begin {theorem}
Let $x\in X_{GL(n)}^{ir}$. Assume that $\Psi_{GL(n-1)}(\zeta_{GL(n)}(x))\in
H^{sr}_{GL(n-1)}$.  Then the
restriction of $V_x(GL(n))$ to $U_q(GL(n-1))$ is semisimple and
 $$ V_x(GL(n))|_{U_q(GL(n-1))}\simeq \bigoplus_{y\in \calM_x}V_y(GL(n-1)).$$
\end{theorem}
\begin{proof} Both
statements follow immediately from Theorem \ref{branch} once we remark
that in this case $$ |\Phi^+|-|\Phi^+_M|-\dim T_M=
\binom{n}{2}-\binom{n-1}{2}-(n-1)=0.$$\end{proof}

We can of course iterate this process restricting first to
$U_q(GL(n-1))$ then to $U_q(GL(n-2))$ and so on. Since $U_q(GL(1))$ is
a polynomial ring in one variable, hence commutative, we deduce that
there is a dense open set in $X_0\subset X_{GL(n)}$, whose simple definition
we leave to the reader, with the property that if $x\in X_0$ then
$V_x(GL(n))$ has a standard decomposition into a direct sum of one
dimensional subspaces. This phenomenon is a counter part for quantum
groups of what we have seen at the end of section \ref{sez:app1} and
it is analogous to the Gelfand-Zeitlin phenomenon.




\def\cprime{$'$} \def\cprime{$'$} \def\cprime{$'$}

\end{document}